\documentclass[11pt]{amsart}
\usepackage{amssymb,amsthm}
\usepackage{verbatim}

\numberwithin{equation}{section}

\newtheorem{lemma}[equation]{Lemma}
\newtheorem{thm}[equation]{Theorem}

\newtheorem{cor}[equation]{Corollary}
\newtheorem{prop}[equation]{Proposition}
\newtheorem{question}[equation]{Question}

\newtheorem{claim}[equation]{Claim}

\theoremstyle{definition}
\newtheorem{defi}[equation]{Definition}

\theoremstyle{remark}

\newtheorem*{remark}{Remark}
\newtheorem*{notation}{Notation}

\renewcommand{\bar}[1]{#1\llap{$\overline{\phantom{\rm#1}}$}}
\DeclareMathOperator{\Gal}{{Gal}}
\DeclareMathOperator{\End}{{End}}
\newcommand{\N}{{\mathbb N}}
\newcommand{\Z}{{\mathbb Z}}
\newcommand{\Q}{{\mathbb Q}}
\newcommand{\Qbar}{{\bar\Q}}
\newcommand{\Kbar}{{\bar K}}
\newcommand{\Ebar}{\bar{E}}
\newcommand{\Fbar}{{\bar F}}
\newcommand{\C}{{\mathbb C}}
\newcommand{\A}{{\mathbb A}}
\newcommand{\G}{{\mathbb G}}
\newcommand{\OO}{{\mathcal O}}

\newcommand{\J}{{\mathcal J}}
\newcommand{\scL}{{\mathcal L}}
\newcommand{\hh}{\widehat{h}}
\newcommand{\tth}{^{\operatorname{th}}}
\newcommand{\iter}[1]{^{\langle #1\rangle}}
\newcommand{\iterp}[1]{\iter{#1}}

\begin{document}



\title{Linear relations between polynomial orbits}

\author{Dragos Ghioca}

\address{
Dragos Ghioca \\
Department of Mathematics \& Computer Science\\
University of Lethbridge \\
Lethbridge, AB T1K 3M4\\
Canada 
}

\email{dragos.ghioca@uleth.ca}

\author{Thomas J. Tucker}

\address{
Thomas Tucker\\
Department of Mathematics\\
Hylan Building\\
University of Rochester\\
Rochester, NY 14627\\
USA
}

\email{ttucker@math.rochester.edu}

\author{Michael E. Zieve}
\address{
Michael E. Zieve\\
Department of Mathematics\\
Hill Center--Busch Campus\\
Rutgers, The State University of New Jersey\\
110 Frelinghuysen Road\\
Piscataway, NJ 08854--8019\\
USA
}
\email{zieve@math.rutgers.edu}
\urladdr{www.math.rutgers.edu/$\sim$zieve}

\begin{abstract}
We study the orbits of a polynomial $f\in\C[X]$, namely the sets
$\{\alpha, f(\alpha), f(f(\alpha)), \dots\}$ with $\alpha\in\C$.
We prove that if two nonlinear complex polynomials $f,g$ have
orbits with infinite intersection, then $f$ and $g$ have a common iterate.
More generally, we describe the intersection of any line in $\C^d$ with a
$d$-tuple of orbits of nonlinear polynomials, and we formulate a question which
generalizes both this result and the Mordell--Lang conjecture.
\end{abstract}

\date{\today}

\maketitle


\section{Introduction}
\label{intro}

One of the main topics in complex dynamics is the behavior of complex numbers
$x$ under repeated application of a polynomial
$f\in\C[X]$.  The basic object of study is the orbit
$\OO_f(x):=\{x,f(x),f(f(x)),\dots\}$.  The theme of many results is that
there are hidden interactions between different orbits of a polynomial $f$:
for\ instance, the crude geometric shape of all orbits is determined by
the orbits of critical points \cite[\S 9]{Blanchard}.
However, the methods of complex dynamics
say little about the interaction between orbits of distinct polynomials.
In this paper we determine when two such orbits have infinite intersection.

\begin{thm}
\label{mainthm}
Pick $x,y\in\C$ and nonlinear $f,g\in\C[X]$.
If $\OO_f(x)\cap\OO_g(y)$ is infinite, then $f$ and $g$ have a
common iterate.
\end{thm}

Here the $n\tth$ iterate $f\iter{n}$ of $f$ is defined as the $n\tth$ power of
$f$ under the operation $a(X)\circ b(X):=a(b(X))$.  We say $f$ and $g$ have a
common iterate if $f\iter{n}=g\iter{m}$ for some $n,m>0$.  Note that if
$f,g\in\C[X]$ have a common iterate, and $\OO_f(x)$ is infinite, then
$\OO_f(x)\cap\OO_g(y)$ is infinite whenever it is nonempty.  The polynomials
$f,g$ with a common iterate were determined by Ritt~\cite{Rittit}: up to
composition with linears, $f$ and $g$ must themselves be iterates of a common
polynomial $h\in\C[X]$ (for a more precise formulation see
Proposition~\ref{Rittit}).  The nonlinearity hypothesis in Theorem~\ref{mainthm}
cannot be removed, since for instance $\OO_{X+1}(0)$ contains $\OO_{X^2}(2)$.

In our previous paper \cite{lines}, we proved Theorem~\ref{mainthm} in the
special case $\deg(f)=\deg(g)$.  In the present paper we prove
Theorem~\ref{mainthm} by combining the result from \cite{lines}
with several new ingredients.

We can interpret Theorem~\ref{mainthm} as describing when the Cartesian product
$\OO_f(x)\times\OO_g(y)$ has infinite intersection with the diagonal
$\Delta:=\{(z,z):z\in\C\}$.  The conclusion says that this occurs just when
there exist positive integers $n,m$ such that $\Delta$ is preserved by the map
$(f\iter{n},g\iter{m})\colon\C^2\to\C^2$ defined by $(z_1,z_2)\mapsto
(f\iter{n}(z_1),g\iter{m}(z_2))$.  Our next result generalizes this to
products of more than two orbits:

\begin{thm}
\label{manyvariablethm}
Let $d$ be a positive integer, let $x_1,\dots,x_d\in\C$,
let $L$ be a line in\/ $\C^d$, and let
$f_1, \dots, f_d\in \C[X]$ satisfy $\deg(f_i) > 1$ for $i=1, \dots, d$.
If the Cartesian product $\OO_{f_1}(x_1)\times\dots\times\OO_{f_d}(x_d)$
has infinite intersection with $L$,
then there are nonnegative integers $m_1, \dots, m_d$ such that
$\sum_{i=1}^d m_i > 0$ and
\[
(f_1\iter{m_1}, \dots, f_d\iter{m_d})(L) = L.
\]
\end{thm}

When Theorem~\ref{manyvariablethm} applies, we can describe the intersection of
$L$ with the product of orbits.  Our description involves the following more
general notion of orbits:
\begin{defi}
If $\Omega$ is a set and $T$ is a set of maps $\Omega\to \Omega$, then for
$\omega\in \Omega$ the \emph{orbit} of $\omega$ under $T$ is
$\OO_T(\omega):=\{t(\omega):t\in T\}$.
\end{defi}

Recall that a \emph{semigroup} is a set with an associative binary relation;
in this paper, all semigroups are required to contain an identity element.
Thus, for $f\in\C[X]$ and $\omega\in\C$, the orbit $\OO_f(\omega)$ equals
$\OO_S(\omega)$ where $S$ is the cyclic semigroup $\langle f\rangle$
generated by the map $f\colon\C\to\C$; in general, if $S=\langle\Phi\rangle$,
then we write $\OO_{\Phi}(\alpha)$ in place of $\OO_S(\alpha)$.
Theorem~\ref{manyvariablethm} enables us to describe the intersection of
a line and a product of orbits:

\begin{cor} \label{manyvariablecor}
Let $\alpha\in \C^d$, let
$f_1, \dots, f_d\in \C[X]$ satisfy $\deg(f_i) > 1$ for $i=1, \dots, d$,
and let $L$ be a line in\/ $\C^d$.  Let $S$ be the semigroup generated by the
maps $\rho_i\colon\C^d\to\C^d$ with $1\le i\le d$, where
$\rho_i$ acts as the identity on each coordinate of $\C^d$ except the $i\tth$,
on which it acts as $f_i$.  Then the intersection of $\OO_S(\alpha)$ with $L$
is $\OO_T(\alpha)$, where $T$ is the union of finitely many cosets of cyclic
subsemigroups of $S$.
\end{cor}

It is natural to seek analogues of Corollary~\ref{manyvariablecor} for other
semigroups of endomorphisms of a variety.  In the following question we write
$\N_0$ for the set of nonnegative integers.

\begin{question}
\label{semigroup}
Let $X$ be a variety defined over\/ $\C$, let $V$ be a closed
subvariety of $X$, let $S$ be a finitely generated commutative subsemigroup of
$\End X$, and let $\alpha\in X(\C)$.  Do the following hold?
\begin{enumerate}
\item[(a)]  The intersection $V\cap\OO_S(\alpha)$ can
be written as $\OO_T(\alpha)$ where $T$ is the union of at most finitely many
cosets of subsemigroups of $S$.
\item[(b)]  For any choice of generators $\Phi_1,\dots,\Phi_r$ of $S$,
let $Z$ be the set of tuples
$(n_1,\dots,n_r)\in{\N_0}^r$
for which $\Phi_1^{n_1}\cdots\Phi_r^{n_r}(\alpha)$ lies in $V$; then $Z$ is the
intersection of\/ ${\N_0}^r$ with a finite union of cosets of subgroups of
$\Z^r$.
\end{enumerate}
\end{question}

Corollary~\ref{manyvariablecor} provides just the third known setting
in which part (a) holds.
In this case part (b) holds as well, and in fact we know no example where
(a) holds but (b) fails (it is not difficult to show that (b) implies (a)).
The first setting in which (a) (and (b)) was shown to hold is when $V$ is a
semiabelian variety and $S$ consists of translations: this is a reformulation
of the Mordell--Lang conjecture, which was proved by Faltings~\cite{Faltings}
and Vojta~\cite{V1} (we will discuss this further in Section~\ref{S:ML}).
Finally, when $S$ is cyclic, it is known that (a) and (b) hold in various
cases \cite{Bell,Jason,PR,Denis,p-adic,lines}, and we expect them to hold
whenever $S$ is cyclic.  We emphasize that the methods used to resolve
Question~\ref{semigroup} in these three settings are completely different
from one another.

In Section~\ref{S:ML} we will present several
examples in which (a) does not hold;  we do not
know any general conjecture predicting when it should hold.
We will also explain how Question~\ref{semigroup} relates to the existence
of positive-dimensional subvarieties of $V$ that are invariant under a
nonidentity endomorphism in $S$.

In case $S=\langle\Phi\rangle$ is cyclic,
Question~\ref{semigroup} fits into Zhang's far-reaching system
of dynamical conjectures \cite{ZhangLec}.  Zhang's conjectures include
dynamical analogues of the Manin-Mumford and Bogomolov conjectures for
abelian varieties (now theorems of Raynaud \cite{Raynaud1, Raynaud2},
Ullmo \cite{Ullmo}, and Zhang \cite{Zhang}), as well as a conjecture
on the existence of a Zariski dense orbit for a large class of endomorphisms
$\Phi$.  Let $Y$ denote the union of the proper subvarieties of $X$ which are
preperiodic under $\Phi$.   Then \cite[Conj.~4.1.6]{ZhangLec} asserts that
$X\ne Y$ if $X$ is an irreducible projective variety and $\Phi$ admits a
polarization; the conclusion of Question~\ref{semigroup}
implies that $\OO_{\Phi}(\alpha)\cap V$ is finite whenever
$\alpha\in X(\C)\setminus Y(\C)$ and $V$ is a proper closed subvariety of $X$.
For more details, see Section~\ref{S:ML}.

In our previous paper \cite{lines}, we proved Theorem~\ref{mainthm} in case
$\deg(f)=\deg(g)$.  The proof went as follows.  First we used a specialization
argument to show it suffices to prove the result when $f,g,x,y$ are all
defined over a number field $K$.  Then in fact they
are defined over some ring $A$ of $S$-integers of $K$, where $S$ is a finite
set of primes; this implies that $\OO_f(x)$ and $\OO_g(y)$ lie in $A$.
Thus, for each $n$, the equation $f\iter{n}(X)=g\iter{n}(Y)$ has infinitely
many solutions in $A\times A$, so by Siegel's theorem the polynomial
$f\iter{n}(X)-g\iter{n}(Y)$ has an
absolutely irreducible factor in $K[X,Y]$ which has genus zero and has at most
two points at infinity.  A result of
Bilu and Tichy describes the
polynomials $F,G\in K[X]$ for which $F(X)-G(Y)$ has such a factor.  This gives
constraints on the shape of $f\iter{n}$ and $g\iter{n}$; by combining the
information deduced for different values of $n$, and using elementary results
about polynomial decomposition, we deduced that either $f$ and $g$ have a
common iterate, or there is a linear $\ell\in K[X]$ such that
$(\ell\circ f\circ\ell\iter{-1},\ell\circ g\circ\ell\iter{-1})=
 (\alpha X^r,\beta X^r)$.  Finally, we proved the result directly for
this last type of polynomials $f,g$.

We use two approaches to prove versions of Theorem~\ref{mainthm} in case
$\deg(f)\ne\deg(g)$, both of which rely on the fact that the result is
known when $\deg(f)=\deg(g)$.  Our first approach utilizes canonical
heights to reduce the problem to the case $\deg(f)=\deg(g)$ treated in
\cite{lines}; this approach does not work when $f,g,x,y$ are defined over
a number field, but works in essentially every other situation (cf.\
Theorem~\ref{first extension}).  Our second
approach uses delicate results about polynomial decomposition in order
to obtain the full Theorem~\ref{mainthm}.  In this proof we do not
use the full strength of the result from \cite{lines}; instead we just use
the main polynomial decomposition result from that paper.
In particular, our proof of Theorem~\ref{mainthm}
does not depend on the complicated specialization argument used in
\cite{lines}.  We now describe the second approach in more detail.

Our proof of Theorem~\ref{mainthm} uses a similar strategy to that in
\cite{lines}, but here the
polynomial decomposition work is much more difficult.  The main reason for
this is that, when analyzing functional equations involving $f\iter{n}$ and
$g\iter{n}$ in case
$\deg(f)=\deg(g)$, we could use the fact that if $A,B,C,D\in \C[X]\setminus\C$
satisfy $A\circ B=C\circ D$ and $\deg(A)=\deg(C)$, then $C=A\circ\ell$ and
$D=\ell\iter{-1}\circ B$
for some linear $\ell\in \C[X]$.  When $f$ and $g$ have distinct degrees, one
must use a different approach.  Our proof relies on the full strength of the
new description given in \cite{MZ} for the collection of all
decompositions of a polynomial; in addition, we use several new types of
polynomial decomposition arguments in the present paper.
As above, for every $m,n$ we find that
$f\iter{n}(X)-g\iter{m}(Y)$ has a genus-zero factor with at most two points
at infinity.
We show that this implies that either $f$ and $g$ have a common iterate, or
there is a linear $\ell\in \C[X]$ such that
$(\ell\circ f\circ\ell\iter{-1},\ell\circ g\circ\ell\iter{-1})$ is
either $(\alpha X^r,\beta X^s)$ or $(\pm T_r,\pm T_s)$, where $T_r$ is
the degree-$r$ Chebychev polynomial of the first kind.
We then use a consequence of Siegel's theorem to handle these
last possibilities.

The contents of this paper are as follows.  In the next section we
state the results of Siegel and Bilu--Tichy, and deduce some consequences.
In Section~\ref{sec decomposition} we present the results about polynomial
decomposition used in this paper.
In the following two sections we prove that if $f,g\in\C[X]$ with
$\deg(f),\deg(g)>1$ are such that, for every $n,m>0$,
$f\iter{n}(X)-g\iter{m}(Y)$ has a genus-zero factor with at most two points
at infinity, then
either $f$ and $g$ have a common iterate or some linear $\ell\in\C[X]$ makes
$(\ell\circ f\circ\ell\iter{-1},\ell\circ g\circ\ell\iter{-1})$ have the form
$(\alpha X^r,\beta X^s)$ or $(\pm T_r,\pm T_s)$.  Then in
Section~\ref{sec conclude proof} we conclude the proof of
Theorem~\ref{mainthm},
and in Section~\ref{sec multivariate} we prove Theorem~\ref{manyvariablethm}
and Corollary~\ref{manyvariablecor}.
In the next several sections we give an alternate proof of
Theorem~\ref{mainthm} in case $x,y,f,g$ cannot be defined over a number field;
this proof uses canonical heights to reduce the problem to the case
$\deg(f)=\deg(g)$ treated in our previous paper, and does not rely on any
difficult polynomial decomposition arguments.  In the final section we
discuss related problems.

\begin{notation}
Throughout this paper, $f\iter{n}$ denotes the $n\tth$ iterate of the
polynomial $f$, with the convention $f\iter{0}=X$.  When $f$ has degree $1$, we
denote the functional inverse of
$f$ by $f\iterp{-1}$; this is again a linear polynomial.
By $T_n$ we mean the (normalized) degree-$n$ Chebychev polynomial of the first
kind, which is defined by the equation $T_n(X+X^{-1})=X^n+X^{-n}$; the
classical Chebychev polynomial $C_n$ defined by $C_n(\cos\theta)=\cos n\theta$
satisfies $2C_n(X/2) = T_n(X)$.  We write $\N$ for
the set of positive integers and $\N_0$ for the set of nonnegative integers.
We write $\overline{K}$ for an algebraic
closure of the field $K$.  We say that $\Phi(X,Y)\in K[X,Y]$ is absolutely
irreducible if it is irreducible in $\overline{K}[X,Y]$.  In this case we let
$C$ be the completion of the normalization of the curve $\Phi(X,Y)=0$, and
define the genus of $\Phi(X,Y)$ to be the (geometric) genus of $C$.  Likewise
we define the points at infinity on $\Phi(X,Y)$ to be the points in
$C(\overline{K})$ which correspond to places of $\Kbar(C)$ extending the
infinite place of $\Kbar(X)$.  In this paper, all subvarieties are closed.
\end{notation}


\section{Integral points on curves}
\label{sec bilu}

The seminal result on curves with infinitely many integral points is the 1929
theorem of Siegel \cite{Siegel}; we use the following generalization due
to Lang \cite[Thm.\ 8.2.4 and 8.5.1]{Lang_diophantine}:

\begin{thm}
\label{siegel thm}
Let $K$ be a finitely generated field of characteristic zero, and let $R$
be a finitely generated subring of $K$.  Let $C$ be a smooth, projective,
geometrically irreducible curve over $K$, and let $\phi$ be a non-constant
function in $K(C)$.  Suppose there are infinitely many points $P\in C(K)$
which are not poles of $\phi$ and which satisfy $\phi(P)\in R$.  Then
$C$ has genus zero and $\phi$ has at most two distinct poles.
\end{thm}

We will use this result in two ways.  One is in the form of the following
consequence due to Lang \cite{Lang_int}.

\begin{cor}
\label{lang thm}
Let $a,b\in\C^*$, and let $\Gamma$ be a finitely generated subgroup of
$\C^*\times\C^*$.  Then the equation $ax+by=1$ has at most finitely
many solutions $(x,y)\in\Gamma$.
\end{cor}

This result is proved by applying Theorem~\ref{siegel thm} to the
genus-$1$ curves $a\alpha X^3 + b\beta Y^3=1$, where $(\alpha,\beta)$ runs
through a finite subset of $\Gamma$ which surjects onto $\Gamma/\Gamma^3$.

To describe the other way we apply Theorem~\ref{siegel thm}, we introduce the
following terminology:

\begin{defi}\label{defexc}
A \emph{Siegel polynomial} over a field $K$ is an absolutely irreducible
polynomial $\Phi(X,Y)\in K[X,Y]$ for which the curve $\Phi(X,Y)=0$ has genus
zero and has at most two points at infinity.  A \emph{Siegel factor} of a
polynomial $\Psi(X,Y)\in K[X,Y]$ is a factor of $\Psi$ which is a
Siegel polynomial over $K$.
\end{defi}

\begin{remark}
What we call Siegel polynomials were called exceptional polynomials in
\cite{BT}; the term `exceptional polynomial' has been used with a different
meaning in several papers (e.g., \cite{GRZ}).
\end{remark}

\begin{remark}
Clearly a Siegel polynomial over $K$ maintains the Siegel property over $\Kbar$.
Further, an irreducible $\Phi\in \Kbar[X,Y]$ is a Siegel polynomial if and only
if $\Phi(\phi,\psi)=0$ for some Laurent polynomials $\phi,\psi\in\Kbar(Z)$
which are not both constant (recall that the \emph{Laurent polynomials} in
$\Kbar(Z)$ are the elements of the form $F/Z^n$ with $F\in\Kbar[Z]$ and
$n\in\N_0$).  We do not know a reference for this fact, so we sketch the proof.
If $\Phi$ is a Siegel polynomial then the function field
of the curve $\Phi(X,Y)=0$ (over $\Kbar$) has the form $\Kbar(Z)$, so
$X=\phi(Z)$ and $Y=\psi(Z)$ for some $\phi,\psi\in\Kbar(Z)$; then
$\Phi(\phi,\psi)=0$ and $\phi,\psi$ are not both constant.
Since $\Phi(X,Y)=0$ has at most two points at infinity, at most two points of
$\Kbar\cup\{\infty\}$ are poles of either $\phi$ or $\psi$.  By making a
suitable linear fractional change to $Z$, we may assume that $\phi$ and $\psi$
have no poles except possibly $0$ and $\infty$, which implies $\phi$ and $\psi$
are Laurent polynomials.  Conversely, suppose $\Phi(\phi,\psi)=0$ for some
Laurent polynomials $\phi,\psi\in\Kbar(Z)$ which are not both constant.
Then the function field of $\Phi(X,Y)=0$ is a subfield $F$ of $\Kbar(Z)$, and
each infinite place of $F$ lies under either $Z=0$ or $Z=\infty$, so indeed
$F$ has genus zero with at most two points at infinity.
\end{remark}

\begin{cor}
\label{siegel}
Let $R$ be a finitely generated integral domain of characteristic zero,
let $K$ be the field of fractions of $R$, and pick $\Phi(X,Y)\in K[X,Y]$.
Suppose there are infinitely many pairs $(x,y)\in R\times R$ for which
$\Phi(x,y)=0$.  Then $\Phi(X,Y)$ has a Siegel factor over $K$.
\end{cor}

\begin{proof}
The hypotheses imply that $\Phi(X,Y)$ has an irreducible factor
$\Psi(X,Y)$ in $\Kbar[X,Y]$ which has infinitely many roots in $R\times R$.
By replacing $\Psi$ by a scalar multiple,
we may assume that some coefficient of $\Psi$ equals $1$.  Since any
$\sigma\in\Gal(\Kbar/K)$ fixes $\Phi$, the polynomial $\Psi^{\sigma}$ is an
absolutely irreducible factor of $\Phi$.  Moreover, every root of $\Psi$
in $R\times R$ is also a root of $\Psi^{\sigma}$; since there are infinitely
many such roots, it follows (e.g., by Bezout's theorem) that $\Psi^{\sigma}$
is a scalar multiple of $\Psi$.  But since $\Psi$ has a coefficient
equal to $1$, the corresponding coefficient of $\Psi^{\sigma}$ is also $1$,
so $\Psi^{\sigma}=\Psi$.  Thus $\Psi$ is fixed by $\Gal(\Kbar/K)$, so
$\Psi\in K[X,Y]$, whence $\Psi$ is the desired Siegel factor.
\end{proof}

In light of Siegel's theorem, there has been intensive study of polynomials
$\Phi(X,Y)$ having a Siegel factor.  As noted above, a nonzero polynomial
$\Phi\in\Kbar[X,Y]$ has a Siegel factor
if and only if $\Phi(\phi,\psi)=0$ for some Laurent polynomials
$\phi,\psi\in\Kbar(X)$ which are not both constant.
Especially strong results have been obtained in case $\Phi(X,Y)=F(X)-G(Y)$ with
$F,G\in K[X]$; in this case the problem amounts to solving the
functional equation $F\circ\phi=G\circ\psi$ in polynomials $F,G\in K[X]$ and
Laurent polynomials $\phi,\psi\in\Kbar[X]$.  Using Ritt's classical results
on such functional equations, together with subsequent results of
Fried and Schinzel (as well as several new ideas),
Bilu and Tichy \cite[Thm.\ 9.3]{BT} proved the following definitive result in
this case.

\begin{thm}
\label{BTthm}
Let $K$ be a field of characteristic zero, and pick $F,G\in K[X]$ for which
$F(X)-G(Y)$ has a Siegel factor in $K[X,Y]$.
Then $F=E\circ F_1\circ \mu$ and
$G=E\circ G_1\circ \nu$, where $E,\mu,\nu\in K[X]$ with
$\deg(\mu)=\deg(\nu)=1$, and either $(F_1,G_1)$ or $(G_1,F_1)$ is one of the
following pairs
(in which $m,n\in\N$,\, $a,b\in K^*$, and $p\in K[X]\setminus\{0\}$):
\begin{enumerate}
\item[(\thethm.1)] $(X^m,\, aX^r p(X)^m)$ with $r$ a nonnegative integer coprime
                    to $m$;
\item[(\thethm.2)] $(X^2,\, (aX^2+b)p(X)^2)$;
\item[(\thethm.3)] $(D_m(X,a^n),\, D_n(X,a^m)$ with $\gcd(m,n)=1$;
\item[(\thethm.4)] $(a^{-m/2}D_m(X,a),\, -b^{-n/2}D_n(X,b))$ with $\gcd(m,n)=2$;
\item[(\thethm.5)] $((aX^2-1)^3,\, 3X^4-4X^3)$;
\item[(\thethm.6)] $(D_m(X,a^{n/d}),\,-D_n(X\cos(\pi/d),a^{m/d}))$ where
                    $d=\gcd(m,n)\ge 3$ and $\cos(2\pi/d)\in K$.
\end{enumerate}
\end{thm}
Here $D_n(X,Y)$ is the unique polynomial in $\Z[X,Y]$ such that
$D_n(U+V,UV)=U^n+V^n$.
Note that, for $\alpha\in K$, the polynomial
$D_n(X,\alpha)\in K[X]$ is monic of degree $n$.  The
defining functional equation implies that $D_n(X,0)=X^n$ and
$\alpha^n D_n(X,1) = D_n(\alpha X,\alpha^2)$ for $\alpha\in\C^*$.
Since $T_n(u+u^{-1})=u^n+u^{-n}$, we have
\begin{equation} \label{TD}
D_n(\alpha X,\alpha^2) = \alpha^n T_n(X) \quad\text{for any $n\in\N$ and
$\alpha\in\C^*$}.
\end{equation}

For our application to orbits of complex polynomials, we will only need
the case $K=\C$ of Theorem~\ref{BTthm}.  We now state a simpler version
of the result in this case.

\begin{cor}\label{BTcor}
For nonconstant $F,G\in\C[X]$, if $F(X)-G(Y)$ has a Siegel factor
in\/ $\C[X,Y]$ then $F=E\circ F_1\circ \mu$ and
$G=E\circ G_1\circ \nu$, where $E,\mu,\nu\in \C[X]$ with
$\deg(\mu)=\deg(\nu)=1$, and either $(F_1,G_1)$ or $(G_1,F_1)$ is one of the
following pairs
(in which $m,n\in\N$ and $p\in \C[X]\setminus\{0\}$):
\begin{enumerate}
\item[(\thethm.1)] $(X^m,\, X^r p(X)^m)$, where $r\in\N_0$ is coprime to
                    $m$;
\item[(\thethm.2)] $(X^2,\, (X^2+1)p(X)^2)$;
\item[(\thethm.3)] $(T_m,\, T_n)$ with $\gcd(m,n)=1$;
\item[(\thethm.4)] $(T_m,\,-T_n)$ with $\gcd(m,n)>1$;
\item[(\thethm.5)] $((X^2-1)^3,\,3X^4-4X^3)$.
\end{enumerate}
\end{cor}

\begin{proof}
Let $E,F_1,G_1,\mu,\nu$ satisfy the conclusion of
Theorem~\ref{BTthm}.  In light of (\ref{TD}), if a pair $(f,g)$ has the
form of one of (\ref{BTthm}.1)--(\ref{BTthm}.6), then there are linear
$\ell_i\in\C[X]$ for which
$(\ell_1\circ f\circ\ell_2,\,\ell_1\circ g\circ\ell_3)$ has the form of
one of (\ref{BTcor}.1)--(\ref{BTcor}.5).  This implies that $(F,G)$ has
the desired form, since we can replace $E$ by $E\circ\ell_1$ and
replace $(\mu,\nu)$ by either $(\ell_2\circ\mu,\,\ell_3\circ\nu)$
or $(\ell_3\circ\mu,\,\ell_2\circ\nu)$.
\end{proof}

\begin{remark}
The converse of Corollary~\ref{BTcor} is also true; since it is not
needed for the present paper, we only sketch the proof.  It suffices to show
that, for each pair $(f,g)$ satisfying one of (\ref{BTcor}.1)--(\ref{BTcor}.5),
we have $f\circ\phi=g\circ\psi$ for some Laurent polynomials
$\phi,\psi\in\C(X)$ which are not both constant.  For this, observe that
\begin{align*}
&X^m \circ X^r p(X^m) = X^r p(X)^m \circ X^m; \\
&X^2 \circ (X+(4X)^{-1})\,p(X-(4X)^{-1}) = 
 (X^2+1)\,p(X)^2\circ (X-(4X)^{-1}); \\
&T_m \circ T_n = T_n \circ T_m; \\
&T_m \circ (X^n+X^{-n}) = -T_n\circ ((\zeta X)^m+(\zeta X)^{-m})
\quad\text{where $\zeta^{mn}=-1$;\,\, and} \\
&(X^2-1)^3 \circ \frac{X^2+2X+X^{-1}-(2X)^{-2}}{\sqrt{3}} =\\
&\quad =(3X^4-4X^3) \circ \frac{(X+1-(2X)^{-1})^3+4}{3}.
\end{align*}
\end{remark}

\begin{remark}
Our statement of Theorem~\ref{BTthm} differs slightly from
\cite[Thm.\ 9.3]{BT},
since there is a mistake in the definition of specific pairs in \cite{BT}
(the terms $a^{m/d}$ and $a^{n/d}$ should be interchanged).  The proof of
\cite[Thm.\ 9.3]{BT} contains some minor errors related to this point, but they
are easy to correct.  Also, although the sentence in \cite{BT} following the
definition of specific pairs is false for odd $n$ (because implication (9) is
false for odd $n$), neither this nor (9) is used in the paper \cite{BT}.
\end{remark}


\section{Polynomial decomposition}
\label{sec decomposition}

Our proof relies on several results about decompositions of polynomials.
Especially, we make crucial use of the following result proved in the
companion paper \cite[Thm.~1.4]{MZ}:

\begin{thm}\label{iterates}
Pick $f\in\C[X]$ with $\deg(f)=n>1$, and suppose there is no linear
$\ell\in\C[X]$ such that $\ell\circ f\circ\ell\iter{-1}$ is either
$X^n$ or $T_n$ or $-T_n$.  Let $r,s\in\C[X]$ and $d\in\N$
satisfy $r\circ s=f\iter{d}$.  Then we have
\begin{align*}
r&=f\iter{i}\circ R \\
s&=S\circ f\iter{j} \\
R\circ S&=f\iter{k}
\end{align*}
where $R,S\in\C[X]$ and $i,j,k\in\N_0$ with $k\le\log_2(n+2)$.
\end{thm}

The proof of this result relies on the full strength of the new description
given in \cite{MZ} for the collection of all decompositions of a polynomial;
this in turn depends on the classical results of Ritt \cite{Ritt} among
other things.  By contrast, all the other polynomial decomposition results we
need can be proved fairly quickly from first principles.

The next result follows from results of Engstrom \cite{Engstrom}; for a proof
using methods akin to Ritt's \cite{Ritt}, see \cite[Cor.~2.9]{MZ}.

\begin{lemma}\label{multipledeg}
Pick $a,b,c,d\in\C[X]\setminus\C$ with $a\circ b = c\circ d$.
If $\deg(c)\mid\deg(a)$, then $a=c\circ t$ for some $t\in\C[X]$.
If $\deg(d)\mid\deg(b)$, then $b=t\circ d$ for some $t\in\C[X]$.
\end{lemma}

We will often use the above two results in conjunction with one another:

\begin{cor}\label{corit}
Pick $f\in\C[X]$ with $\deg(f)=n>1$, and assume there is no linear
$\ell\in\C[X]$ such that $\ell\circ f\circ\ell\iter{-1}$ is either
$X^n$ or $T_n$ or $-T_n$.  Then there is a finite subset
$S$ of $\C[X]$ such that, if $r,s\in\C[X]$ and $d\in\N$ satisfy
$r\circ s=f\iter{d}$, then
\begin{itemize}
\item either $r=f\circ t$ (with $t\in\C[X]$) or $r\circ\ell\in S$
(with $\ell\in\C[X]$ linear);
\item either $s=t\circ f$ (with $t\in\C[X]$) or $\ell\circ s\in S$
(with $\ell\in\C[X]$ linear).
\end{itemize}
\end{cor}

As an immediate consequence of the functional equation defining $T_n$,
we see that $T_n$ is either an even or odd polynomial:

\begin{lemma}\label{parity}
For any $n\in\N$, we have $T_n(-X)=(-1)^n T_n(X)$.
\end{lemma}

Note that $X^d \circ X^e = X^{de}$ and $T_d\circ T_e = T_{de}$.
By Lemma~\ref{multipledeg}, these are essentially the only decompositions
of $X^n$ and $T_n$:

\begin{lemma}\label{chebdec}
If $n\in\N$ and $f,g\in\C[X]$ satisfy $f\circ g=X^n$,
then $f=X^d\circ\ell$ and $g=\ell\iter{-1}\circ X^{n/d}$
for some linear $\ell\in\C[X]$ and some positive divisor $d$ of $n$.
If $n\in\N$ and $f,g\in\C[X]$ satisfy $f\circ g=T_n$,
then $f=T_d\circ\ell$ and $g=\ell\iter{-1}\circ T^{n/d}$
for some linear $\ell\in\C[X]$ and some positive divisor $d$ of $n$.
\end{lemma}

The following simple result describes the linear relations between polynomials
of the form $X^n$ or $T_n$ \cite[Lemmas 3.13 and 3.14]{MZ}:

\begin{lemma}\label{chebeq}
Pick $n\in\N$ and linear $a,b\in\C[X]$.
\begin{enumerate}
\item[(\thethm.1)] If $n>1$ and $a\circ X^n\circ b = X^n$,
then $b=\beta X$ and $a=X/\beta^n$ for some $\beta\in\C^*$.
\item[(\thethm.2)] If $n>2$ then $a\circ X^n\circ b\ne T_n$.
\item[(\thethm.3)] If $n>2$ and $a\circ T_n\circ b = T_n$, then $b=\epsilon X$
and $a=\epsilon^n X$ for some $\epsilon\in\{1,-1\}$.
\end{enumerate}
\end{lemma}

The previous two results have the following consequence \cite[Cor.~3.10]{MZ}:

\begin{lemma}\label{chebswap}
Pick $r,s\in\Z$ and linear $\ell,\ell_1,\ell_2\in\C[X]$.
If $r,s>1$ and $X^r\circ \ell\circ X^s = \ell_1\circ X^{rs}\circ\ell_2$,
then $\ell=\alpha X$ for some $\alpha\in\C^*$.
If $r,s>2$ and $T_r\circ \ell\circ T_s=\ell_1\circ T_{rs}\circ\ell_2$,
then $\ell = \epsilon X$ for some $\epsilon\in\{1,-1\}$.
\end{lemma}

We also need to know the possible decompositions of polynomials of the
form $X^i h(X)^n$ \cite[Lemma~3.11]{MZ}:

\begin{lemma}\label{decompositions}
If $a\circ b=X^i h(X)^n$ with $h\in\C[X]\setminus\{0\}$ and coprime $i,n\in\N$,
then $a=X^j \hat{h}(X)^n\circ\ell$ and
$b=\ell\iter{-1}\circ X^k \tilde{h}(X)^n$ for some $j,k\in\N$ and some
$\hat{h},\tilde{h},\ell\in\C[X]$ with $\ell$ linear.
\end{lemma}

The following result presents situations where the shape of a polynomial
is determined by the shape of one of its iterates.

\begin{lemma} \label{handlepowers}
Pick $f,\ell,\hat\ell\in\C[X]$ with $r:=\deg(f)>1$ and
$\ell,\hat\ell$ linear, and pick $n\in\Z_{>1}$.
\begin{enumerate}
\item[(\thethm.1)] If $f\iter{n}=\ell\circ X^{r^n}\circ\hat\ell$,
then $f=\ell\circ \alpha X^r\circ\ell\iter{-1}$ for some $\alpha\in\C^*$.
\item[(\thethm.2)] If $f\iter{n}=\ell\circ T_{r^n}\circ\hat\ell$ and
$\{r,n\}\ne\{2\}$, then $f=\ell\circ T_r\circ\epsilon \ell\iter{-1}$ for some
$\epsilon\in\{1,-1\}$.
\end{enumerate}
\end{lemma}

\begin{proof}
If $f\iter{n}=\ell\circ X^{r^n}\circ\hat\ell$, then
$f=\ell\circ X^r\circ\bar\ell$
for some linear $\bar\ell$ (by Lemma~\ref{multipledeg}).
Likewise $f\iter{2}=\ell\circ X^{r^2}\circ \tilde\ell$, so
Lemma~\ref{chebswap} implies that $\bar\ell\circ\ell=\beta X$
for some $\beta\in\C^*$.  Hence $f=\ell\circ X^r\circ\beta \ell\iter{-1}$.

Henceforth suppose $f\iter{n}=\ell\circ T_{r^n}\circ\hat\ell$ and $n>1$.
As above, $f=\ell\circ T_r\circ\bar\ell$
and $f\iter{2}=\ell\circ T_{r^2}\circ\tilde\ell$, so if $r>2$ then
Lemma~\ref{chebswap} implies $\bar\ell\circ\ell=\epsilon X$ for some
$\epsilon\in\{1,-1\}$, whence $f=\ell\circ T_r\circ\epsilon \ell\iter{-1}$.  

Now assume $r=2$ and $n>2$.  Then $f=\ell\circ T_2\circ\bar\ell$ and
$f\iter{3}=\ell\circ T_8\circ\tilde\ell$.
Writing $\ell\iter{-1}\circ f\iter{3}=(T_2\circ\bar\ell\circ\ell)\circ
(T_2\circ\bar\ell\circ\ell)\circ (T_2\circ\bar\ell) = T_2\circ T_2\circ
(T_2\circ\tilde\ell)$, Lemma~\ref{multipledeg} implies there are linears
$\mu,\lambda\in\C[X]$ such that
$T_2\circ\bar\ell=\lambda\iter{-1}\circ T_2\circ\tilde\ell$ and
$T_2\circ\bar\ell\circ\ell = \mu\iter{-1}\circ T_2\circ\lambda$ and
$T_2\circ\bar\ell\circ\ell = T_2\circ\mu$.
Since $T_2=(X-2)\circ X^2$, by Lemma~\ref{chebeq} the equality
$T_2\circ\mu=\mu\iter{-1}\circ T_2\circ\lambda$ implies
that $\mu\circ\lambda\iter{-1}=\beta X$ and
$\mu=-2+(X+2)/\beta^2$ for some $\beta\in\C^*$.
Likewise, from $\lambda\circ T_2\circ\bar\ell\circ\tilde\ell\iter{-1}=T_2$
we get $\lambda=-2+(X+2)/\alpha^2$ for some $\alpha\in\C^*$; but also
$\lambda=\beta^{-1}\mu$, so since $\lambda$ and $\mu$ fix $-2$, it follows
that $\beta=1$.  Thus $\mu=X$, so we have
$T_2\circ\bar\ell\circ\ell=T_2$ and thus $\bar\ell\circ\ell=\epsilon X$
with $\epsilon\in\{1,-1\}$, and the result follows.
\end{proof}

\begin{remark}
The hypothesis $\{r,n\}\ne\{2\}$ is needed in (\ref{handlepowers}.2):
for any linear $\ell$ and any $\alpha\in\C^*\setminus\{1,-1\}$,
the polynomial $f=\ell\circ T_2\circ (-2+\alpha^2(X+2))\circ\ell\iter{-1}$
satisfies $f\iter{2}=
\ell\circ T_4\circ(-2\alpha+\alpha^3(X+2))\circ\ell\iter{-1}$
but $f\ne\ell\circ T_2\circ\pm \ell\iter{-1}$.
\end{remark}

Although it is not used in this paper, for the reader's convenience we
recall Ritt's description of polynomials with a common iterate
\cite[p.~356]{Rittit}:

\begin{prop}[Ritt]
\label{Rittit}
Let $f_1,f_2\in \C[X]$ with
$d_i:=\deg(f_i)>1$ for each $i\in\{1,2\}$.  For $m_1,m_2\in\N$, we have
$f_1\iter{m_1}=f_2\iter{m_2}$ if and only if 
$f_1(X)=-\beta+\epsilon_1 g\iter{n_1}(X+\beta)$ and
$f_2(X)=-\beta+ \epsilon_2 g\iter{n_2}(X+\beta)$ for some $n_1,n_2\in\N$
with $n_1m_1=n_2m_2$, some $g\in X^r \C[X^s]$ (with $r,s\in\N_0$),
and some $\epsilon_1,\epsilon_2,\beta\in \C$ with
$\epsilon_i^s = 1$ and $\epsilon_i^{(d_i^{m_i}-1)/(d_i-1)}=1$ for each
$i\in\{1,2\}$.
\end{prop}

\section{Commensurable polynomials}
\label{sec commensurate}

In this section we analyze $f,g\in\C[X]$ which are \emph{commensurable}, in the
sense that for every $m\in\N$ there exist $n\in\N$ and $h_1,h_2\in\C[X]$
such that $f\iter{n}=g\iter{m}\circ h_1$ and $g\iter{n}=f\iter{m}\circ h_2$.
Plainly two polynomials with a common iterate are commensurable; we give an
explicit description of all other pairs of commensurable polynomials.
In fact, we need only assume half of the commensurability hypothesis:

\begin{prop}\label{gensec4}
Pick $f,g\in \C[X]$ for which $r:=\deg(f)$ and $s:=\deg(g)$ satisfy
$r,s>1$.
Suppose that, for every $m\in\N$, there exists $n\in\N$
and $h\in \C[X]$ such that $g\iter{n}=f\iter{m}\circ h$.  Then either $f$ and
$g$ have a common iterate, or there is a linear $\ell\in \C[X]$ such that
$(\ell\circ f\circ\ell\iter{-1},\ell\circ g\circ\ell\iter{-1})$ is either
$(\alpha X^r,X^s)$ (with $\alpha\in \C^*$) or
$(T_r\circ\hat\epsilon X,T_s\circ\epsilon X)$ (with $\hat\epsilon,
\epsilon\in\{1,-1\}$).
\end{prop}

\begin{remark}
The converse of Proposition~\ref{gensec4} holds if and only if every
prime factor of $r$ is also a factor of $s$.
\end{remark}

Our proof of Proposition~\ref{gensec4} consists of a reduction to the
case $r=s$.  The case $r=s$ of Proposition~\ref{gensec4} was analyzed in
our previous paper \cite{lines}, as one of the main ingredients in
our proof of Theorem~\ref{mainthm} in case $\deg(f)=\deg(g)$. 
The following result is \cite[Prop.\ 3.3]{lines}.

\begin{prop}\label{linearprop}
Let $F,G\in \C[X]$ satisfy
$\deg(F) = \deg(G)=r>1$.  Suppose that, for every $m\in\N$, there is a
linear $\ell_m\in \C[X]$ such that $G\iter{m}=F\iter{m}\circ\ell_m$.
Then either $F$ and $G$ have a common iterate, or there is a linear
$\ell\in\C[X]$ for which $\ell\circ F\circ\ell\iter{-1}=\alpha X^r$ and
$\ell\circ G\circ\ell\iter{-1}=\beta X^r$ with $\alpha,\beta\in\C^*$.
\end{prop}

By Lemma~\ref{multipledeg}, this implies the case $r=s$ of
Proposition~\ref{gensec4}.  Note that Chebychev polynomials are given
special mention in the conclusion of Proposition~\ref{gensec4}, but not in
the conclusion of Proposition~\ref{linearprop}; this is because $T_r(X)$ and
$T_r(-X)$ have the same second iterate.

\begin{proof}[Proof of Proposition~\ref{gensec4}.]
First assume that $\ell\circ g\circ\ell\iter{-1}=X^s$ for some
linear $\ell\in\C[X]$.  Then $g\iter{n}=f\iter{2}\circ h$
becomes $\ell\iter{-1}\circ X^{s^n}\circ\ell=f\iter{2}\circ h$,
so Lemma~\ref{chebdec} implies $f\iter{2}=\ell\iter{-1}\circ X^{r^2}\circ
\hat\ell$ for some linear $\hat\ell\in\C[X]$.  Now Lemma~\ref{handlepowers}
implies $f=\ell\iter{-1}\circ\alpha X^r\circ\ell$ for some
$\alpha\in\C^*$, so the result holds in this case.

Next assume that $\ell\circ g\circ\ell\iter{-1}=T_s\circ\epsilon X$ for some
linear $\ell\in\C[X]$ and some $\epsilon\in\{1,-1\}$.  Then we can use
the fact that $T_s(-X)=(-1)^s T_s(X)$ 
to rewrite $g\iter{n}=f\iter{3}\circ h$ as
$\ell\iter{-1}\circ T_{s^n}\circ\epsilon^n\ell=f\iter{3}\circ h$.  As
above, Lemma~\ref{chebdec} implies that $f\iter{3}=\ell\iter{-1}\circ T_{r^3}
\circ \hat\ell$ for some linear $\hat\ell\in\C[X]$.  Then
Lemma~\ref{handlepowers} implies $f=\ell\iter{-1}\circ T_r\circ
\hat\epsilon \ell$ with $\hat\epsilon\in\{1,-1\}$,
so the result holds in this case.

Henceforth assume there is no linear $\ell\in\C[X]$ for which
$\ell\circ g\circ\ell\iter{-1}$ is either $X^s$ or $T_s$ or $T_s(-X)$.
For $m\in\N$, let $n\in\N$ be minimal for which $g\iter{n}=f\iter{m}\circ h$
with $h\in\C[X]$, and let $h_m\in\C[X]$ satisfy
$g\iter{n}=f\iter{m}\circ h_m$.
Minimality of $n$ implies there is no
$t\in\C[X]$ with $h_m=t\circ g$, so by Corollary~\ref{corit} there is
a bound on $\deg(h_m)$ depending only on $g$.  In particular, this implies
there are distinct $m,M\in\N$ for which $\deg(h_m)=\deg(h_M)$.
Assuming $m<M$ and equating degrees in the identities
$g\iter{n}=f\iter{m}\circ h_m$ and $g\iter{N}=f\iter{M}\circ h_M$,
it follows that $\deg(g)^{N-n}=\deg(f)^{M-m}$.

Let $S=c(M-m)$ with $c\in\N$, and write $g\iter{R}=f\iter{S}\circ h_S$.
Since $h_S\ne t\circ g$ for every $t\in\C[X]$,
Lemma~\ref{multipledeg} implies $\deg(g)\nmid\deg(h_S)$, so we must
have $R=c(N-n)$ and $\deg(h_S)=1$.  Thus, $F:=f\iterp{M-m}$ and
$G:=g\iterp{N-n}$ satisfy the hypotheses of Proposition~\ref{linearprop},
so either $F$ and $G$ have a common iterate (so $f$ and $g$ do as well),
or there is a linear $\ell\in\C[X]$ for which
$\ell\circ G\circ\ell\iter{-1}=\beta X^{\deg(G)}$ (with $\beta\in\C^*$).
In the latter case, Lemma~\ref{handlepowers} implies there is a linear
$\hat\ell\in\C[X]$ such that
$\hat\ell\circ g\circ\hat\ell\iter{-1}=X^s$, contradicting our
assumption on the form of $g$.
\end{proof}


\section{Non-commensurable polynomials}
\label{sec noncommensurate}

In this section we classify the non-commensurable pairs of polynomials $(f,g)$
for which each polynomial $f\iter{n}(X)-g\iter{n}(Y)$ has a Siegel factor
(in the sense of Definition~\ref{defexc}).

\begin{prop}\label{gensec3}
Pick $f,g \in \C[X]$ for which $r:=\deg(f)$ and $s:=\deg(g)$ satisfy $r,s>1$.
Assume there exists $m\in\N$ with these properties:
\begin{enumerate}
\item[(\thethm.1)]  $g\iter{n}\ne f\iter{m}\circ h$ for every $h\in\C[X]$ and
  $n\in\N$; and
\item[(\thethm.2)] there are infinitely many $j\in\N$ for which
$f\iter{mj}(X)-g\iter{mj}(Y)$ has a Siegel factor in $\C[X,Y]$.
\end{enumerate}
Then there is a linear $\ell\in\C[X]$ for which
$(\ell\circ f\circ\ell\iter{-1},\ell\circ g\circ\ell\iter{-1})$
is either $(X^r,\alpha X^s)$ (with $\alpha\in\C^*$) or
$(\epsilon_1 T_r, \epsilon_2 T_s)$ (with $\epsilon_1,\epsilon_2\in\{1,-1\}$).
\end{prop}

\begin{remark}
The converse of Proposition~\ref{gensec3} holds if and only if some prime
factor of $r$ is not a factor of $s$.
\end{remark}

\begin{remark}
The pair $(\epsilon_1 T_r, \epsilon_2 T_s)$ in the conclusion of
Proposition~\ref{gensec3} differs slightly from the pair
$(T_r\circ\hat\epsilon X,T_s\circ\epsilon X)$ in the conclusion of
Proposition~\ref{gensec4}.  The latter pairs are special cases of the
former pairs, but if $r$ and $s$ are even then $(T_r,-T_s)$ cannot be written
in the latter form (even after conjugation by a linear).
\end{remark}

\begin{proof}[Proof of Proposition~\ref{gensec3}.]
Let $\J$ be the (infinite) set of $j\in\N$ for which
$f\iter{mj}(X)-g\iter{mj}(Y)$ has a Siegel factor in $\C[X,Y]$.
For $j\in\J$, Corollary~\ref{BTcor} implies there are $A_j,B_j,C_j\in \C[X]$
and linear $\mu_j,\nu_j\in \C[X]$ such that $f\iter{mj}=A_j\circ B_j\circ\mu_j$
and $g\iter{mj}=A_j\circ C_j\circ\nu_j$,
where either $(B_j,C_j)$ or $(C_j,B_j)$ has the form of one of
(\ref{BTcor}.1)--(\ref{BTcor}.5).

We split the proof into two cases, depending on whether the degrees of
the polynomials $A_j$ are bounded.
\vskip.1in


\noindent\textit{Case 1: $\{\deg(A_j):j\in\J\}$ is infinite}\newline
In this case there is an infinite subset $\J_0$ of $\J$ such that
$j\mapsto\deg(A_j)$ is a strictly increasing function on $\J_0$.
Replacing $\J$ by $\J_0$, it follows that $\deg(A_j)$ exceeds any
prescribed bound whenever $j\in\J$ is sufficiently large.
By (\ref{gensec3}.1), for $j\in\J$ we cannot have $A_j=f\iter{m}\circ h$ with
$h\in\C[X]$.  Applying Corollary~\ref{corit} to the decomposition
$(f\iter{m})\iter{j}=A_j\circ (B_j\circ\mu_j)$, and recalling that
$\deg(A_j)\to\infty$, it follows that for sufficiently large $j$ we have
either
\begin{equation*}
f\iter{mj} = \ell_j\circ X^{r^{mj}}\circ \ell_j\iter{-1}\quad\text{ or }\quad
f\iter{mj}=\ell_j\circ T_{r^{mj}}\circ \epsilon_j \ell_j\iter{-1},
\end{equation*}
where $\ell_j\in\C[X]$ is linear and $\epsilon_j\in\{1,-1\}$.
Thus, by Lemma~\ref{handlepowers}, either
\begin{align}
\label{fcyc} f&=\ell\iter{-1}\circ X^r\circ\ell\quad\text{ or}\\
\label{fcheb} f&=\ell\iter{-1}\circ T_r\circ\epsilon \ell
\end{align}
for some linear $\ell\in\C[X]$ and some $\epsilon\in\{1,-1\}$.
It remains to determine the shape of $g$.  To this end note that, in the
cases (\ref{fcyc}) and (\ref{fcheb}), respectively, we have
\begin{align*}
f\iter{n}&=\ell\iter{-1}\circ X^{r^n}\circ\ell\quad\text{ and}\\
f\iter{n}&=\ell\iter{-1}\circ T_{r^n}\circ \epsilon^n \ell,
\end{align*}
where in the latter case we have used Lemma~\ref{parity}.
Since $f\iter{mj}=A_j\circ (B_j\circ\mu_j)$, Lemma~\ref{chebdec} implies
that for every $j\in\J$ there is a linear $\hat\ell_j\in\C[X]$ such that
\begin{align}
\label{Ajfcyc} A_j&=\ell\iter{-1}\circ X^{\deg(A_j)}\circ\hat\ell_j\quad
  \text{ if (\ref{fcyc}) holds, and}\\
\label{Ajfcheb} A_j&=\ell\iter{-1}\circ T_{\deg(A_j)}\circ\hat\ell_j \quad
  \text{ if (\ref{fcheb}) holds.}
\end{align}

If $A_j=g\iter{3}\circ h$ for some $j\in\J$ and $h\in\C[X]$, then
by Lemma~\ref{chebdec} there is a linear $\tilde\ell\in\C[X]$ such that
\begin{align*}
g\iter{3}&=\ell\iter{-1}\circ X^{s^3}\circ\tilde\ell\quad
  \text{ if (\ref{Ajfcyc}) holds, and}\\
g\iter{3}&=\ell\iter{-1}\circ T_{s^3}\circ\tilde\ell\quad
  \text{ if (\ref{Ajfcheb}) holds.}
\end{align*}
By Lemma~\ref{handlepowers}, there are $\alpha\in\C^*$ and
$\hat\epsilon\in\{1,-1\}$ such that
\begin{align*}
g&=\ell\iter{-1}\circ \alpha X^s\circ\ell\quad
  \text{ if (\ref{fcyc}) holds, and}\\
g&=\ell\iter{-1}\circ T_s\circ\hat\epsilon\ell\quad
  \text{ if (\ref{fcheb}) holds.}
\end{align*}
This completes the proof in case $A_j=g\iter{3}\circ h$.

Now suppose that $A_j\ne g\iter{3}\circ h$ for every $j\in\J$ and
$h\in\C[X]$.  Since
$(g\iter{3})\iter{mj} = g\iter{3mj}=A_j\circ (C_j\circ\nu_j\circ g\iter{2mj})$,
and moreover $\deg(A_j)\to\infty$ as $j\to\infty$, Corollary~\ref{corit}
implies that either
\begin{align*}
g\iter{3} &= \tilde\ell\circ X^{s^3}\circ \tilde\ell\iter{-1}\quad\text{ or}\\
g\iter{3} &= \tilde\ell\circ T_{s^3}\circ\tilde\epsilon \tilde\ell\iter{-1},
\end{align*}
where $\tilde\ell\in\C[X]$ is linear and $\tilde\epsilon\in\{1,-1\}$.
By Lemma~\ref{handlepowers}, either
\begin{align}
\label{gcyc} g &= \tilde\ell\circ \beta X^s\circ\tilde\ell\iter{-1}\quad\text{ or}\\
\label{gcheb} g &= \tilde\ell\circ T_s\circ\hat\epsilon\tilde\ell\iter{-1},
\end{align}
where $\beta\in\C^*$ and $\hat\epsilon\in\{1,-1\}$.  Thus, for $n\in\N$, we have
\begin{align*}
g\iter{n} &= \tilde\ell \circ \beta^{1+s+\dots+s^{n-1}} X^{s^n}\circ
  \tilde\ell\iter{-1}\quad\text{ if (\ref{gcyc}) holds, and}\\
g\iter{n} &= \tilde\ell \circ T_{s^n}\circ\hat\epsilon^n\tilde\ell\iter{-1}\quad
  \text{ if (\ref{gcheb}) holds.}
\end{align*}
Applying Lemma~\ref{chebdec} to the decomposition
$g\iter{mj}=A_j\circ (C_j\circ\nu_j)$, we see that there is
a linear $\tilde\ell_j\in\C[X]$ such that
\begin{align}
\label{Ajgcyc} A_j &= \tilde\ell\circ X^{\deg(A_j)}\circ\tilde\ell_j
  \quad\text{ if (\ref{gcyc}) holds, and}\\
\label{Ajgcheb} A_j &= \tilde\ell\circ T_{\deg(A_j)}\circ\tilde\ell_j
\quad\text{ if (\ref{gcheb}) holds.}
\end{align}
Choose $j\in\J$ for which $\deg(A_j)>2$.  

If (\ref{fcheb}) holds then so does (\ref{Ajfcheb}), so Lemma~\ref{chebeq}
implies (\ref{Ajgcyc}) does not hold,  whence (\ref{Ajgcheb}) and
(\ref{gcheb}) hold; Lemma~\ref{chebeq} implies further that
$\tilde\ell=\ell\iter{-1}\circ\delta X$ for some
$\delta\in\{1,-1\}$.  But then
\begin{align*}
g &= \ell\iter{-1}\circ\delta T_s\circ \hat\epsilon\delta\ell \\
  &= \ell\iter{-1}\circ\delta^{1+s}\hat\epsilon^s T_s\circ\ell,
\end{align*}
which completes the proof in this case.

Finally, if (\ref{fcyc}) holds then so does (\ref{Ajfcyc}), so
Lemma~\ref{chebeq} implies (\ref{Ajgcheb}) does not hold, whence
(\ref{Ajgcyc}) and (\ref{gcyc}) hold; moreover,
$\tilde\ell=\ell\iter{-1}\circ\gamma X$ for some $\gamma\in\C^*$.  But then
\begin{align*}
g &= \ell\iter{-1}\circ \gamma\beta X^s\circ\gamma^{-1}\ell \\
  &= \ell\iter{-1}\circ \gamma^{1-s}\beta X^s\circ\ell,
\end{align*}
which completes the proof in Case 1.
\vskip.2in


\noindent\textit{Case 2: $\{\deg(A_j):j\in\J\}$ is finite.}\newline
Suppose first that $e:=\gcd(\deg(f),\deg(g))$ satisfies $e>1$.
In this case, $\gcd(\deg(f\iter{mj}),\deg(g\iter{mj})) = e^{mj}\to\infty$
as $j\to\infty$, and since $\deg(A_j)$ is bounded it follows that
$\gcd(\deg(B_j),\deg(C_j))\to\infty$.
For any nonconstant $F,G\in\C[X]$ such that $(F,G)$ has any of the forms
(\ref{BTcor}.1)--(\ref{BTcor}.5) other than (\ref{BTcor}.4), we observe that
$\gcd(\deg(F),\deg(G))\le 2$; thus, for every sufficiently large $j\in\J$,
either $(B_j,C_j)$ or $(C_j,B_j)$ has the form (\ref{BTcor}.4).
For any such $j$, after perhaps replacing $(A_j,B_j,C_j)$ by
$(A_j(-X),-B_j,-C_j)$, we find that
$B_j=T_{\deg(B_j)}$ and $C_j=-T_{\deg(C_j)}$.  Since
$f\iter{mj}=A_j\circ T_{\deg(B_j)}\circ\mu_j$ and $\deg(A_j)$ is bounded,
for sufficiently large $j\in\J$ we must have $r^3\mid\deg(B_j)$; applying
Lemma~\ref{multipledeg} to the decomposition
$f\iter{mj-3}\circ f\iter{3} = (A_j\circ T_{\deg(B_j)/r^3})\circ
 (T_{r^3}\circ\mu_j)$
gives $f\iter{3}=\ell_j\circ T_{r^3}\circ\mu_j$
with $\ell_j\in\C[X]$ linear.  Lemma~\ref{handlepowers}
implies $f=\ell_j\circ T_r\circ\epsilon \ell_j\iter{-1}$ with
$\epsilon\in\{1,-1\}$; then
$\ell_j\circ T_{r^3}\circ\mu_j=f\iter{3}=\ell_j\circ T_{r^3}\circ
  \epsilon \ell_j\iter{-1}$,
so Lemma~\ref{chebeq} implies $\mu_j=\delta\epsilon\ell_j\iter{-1}$ for some
$\delta\in\{1,-1\}$ with $\delta^r=1$.  But then
$A_j\circ T_{\deg(B_j)}\circ\mu_j=f\iter{mj}=\mu_j\iter{-1}\circ
  \delta\epsilon T_{r^{mj}}\circ\delta\epsilon^{mj+1}\mu_j$, so
Lemma~\ref{chebdec} implies there is a linear $\tilde\ell\in\C[X]$ such that
$A_j\circ\tilde\ell=\mu_j\iter{-1}\circ\delta\epsilon T_{\deg(A_j)}$ and
$\tilde\ell\iter{-1}\circ T_{\deg(B_j)}\circ\mu_j = T_{\deg(B_j)}\circ
  \delta\epsilon^{mj+1}\mu_j$.
Then $\tilde\ell\in\{X,-X\}$, so $\mu_j \circ A_j = \tilde\epsilon T_{\deg(A_j)}$
with $\tilde\epsilon\in\{1,-1\}$.  The same argument shows that
$\nu_j \circ A_j = \hat\epsilon T_{\deg(A_j)}$ for some $\hat\epsilon\in\{1,-1\}$,
so $\hat\epsilon\nu_j = \tilde\epsilon \mu_j$.  From above,
$f=\mu_j\iter{-1}\circ \epsilon_0 T_r\circ \mu_j$ with $\epsilon_0\in\{1,-1\}$,
and similarly $g=\nu_j\iter{-1}\circ \epsilon_1 T_s\circ\nu_j$ with
$\epsilon_1\in\{1,-1\}$, so $g=\mu_j\iter{-1}\circ\epsilon_2 T_2\circ\mu_j$
with $\epsilon_2\in\{1,-1\}$, and the result follows.

Henceforth suppose that $\gcd(\deg(f),\deg(g))=1$.
In this case, for $j\in\J$ we have $\deg(A_j)=1$ and
$\gcd(\deg(B_j),\deg(C_j))=1$; by examining (\ref{BTcor}.1)--(\ref{BTcor}.5),
we see that one of $(B_j,C_j)$ and $(C_j,B_j)$) must have the
form of either (\ref{BTcor}.1) or (\ref{BTcor}.3).

Suppose there is some $j\in\J$ with $j>2/m$ such that either $(B_j,C_j)$ or
$(C_j,B_j)$ has the form (\ref{BTcor}.3).  For any such $j$ we have
\[
(B_j,C_j)=(T_{\deg(B_j)},T_{\deg(C_j)});
\]
since $A_j$ is linear, this implies
\begin{align*}
f\iter{mj} &= A_j\circ T_{r^{mj}}\circ\mu_j\quad\text{ and}\\
g\iter{mj} &= A_j \circ T_{s^{mj}}\circ\nu_j.
\end{align*}
By Lemma~\ref{handlepowers}, we have
\begin{align*}
f &= A_j \circ T_r\circ\epsilon_j\circ A_j\iter{-1}\quad\text{ and}\\
g &= A_j \circ T_s\circ\bar\epsilon_j\circ A_j\iter{-1}
\end{align*}
for some $\epsilon_j,\bar\epsilon_j\in\{1,-1\}$, so the result holds.

Now suppose that, for every $j\in\J$ with $j>2/m$, either $(B_j,C_j)$ or
$(C_j,B_j)$ has the form (\ref{BTcor}.1).
For any such $j$, we have
\[
\{B_j,\,C_j\} = \{X^n,\, X^i p(X)^n\}
\]
where $p\in\C[X]$ and $i\in\N_0$ satisfy $\gcd(i,n)=1$.
Since $n$ is the degree of either $f\iter{mj}$ or $g\iter{mj}$,
we have $n\in\{r^{mj}, s^{mj}\}$, so $n>1$ and thus $i>0$.
Lemmas~\ref{chebdec} and \ref{decompositions} imply that
\begin{equation}
\label{2ndit}
\{f\iter{2},\,g\iter{2}\} = \{A_j\circ X^{\tilde n}\circ\mu,\,
A_j\circ X^{\tilde i} {\tilde p}(X)^n\circ\nu\}
\end{equation}
where $\tilde{i},\tilde{n}\in\N$ and
$\mu,\nu,\tilde{p}\in\C[X]$ with $\mu,\nu$ linear.
We may assume that $j$ satisfies
\[
\min(r,s)^{mj} > \max(r,s)^2.
\]
Since $n\in\{r^{mj},\, s^{mj}\}$,
it follows that $n>\max(r,s)^2$, so we must have $\tilde{p}\in\C^*$.
Applying Lemma~\ref{handlepowers} to (\ref{2ndit}), we conclude
that
\[
(f,g) = (A_j\circ \hat\alpha X^r\circ A_j\iter{-1},\,
 A_j\circ\hat\beta X^s\circ A_j\iter{-1})
\]
for some $\hat\alpha,\hat\beta\in\C^*$.  Finally, after replacing
$A_j$ by $A_j\circ\gamma X$ for suitable $\gamma\in\C^*$, we may
assume $\hat\alpha=1$, which completes the proof.
\end{proof}


\section{Proof of Theorem~\ref{mainthm}}
\label{sec conclude proof}

In this section we conclude the proof of Theorem~\ref{mainthm}.  Our strategy
is to combine the results of the previous two sections with Siegel's theorem,
in order to reduce to the case that the pair $(f,g)$ has one of the two forms
\begin{align}
\label{cyc} &(X^r,\, \beta X^s),\, \text{ with $\beta\in\C^*$ and
  $r,s\in\Z_{>1}$}; \\
\label{cheb} &(\epsilon_1 T_r,\, \epsilon_2 T_s),\, \text{ with $\epsilon_1,
  \epsilon_2\in\{1,-1\}$ and $r,s\in\Z_{>1}$}.
\end{align}
We then use Corollary~\ref{lang thm} (which is a consequence of Siegel's
theorem) to handle these two possibilities.

\begin{prop}
Pick $f,g\in\C[X]$ for which $r:=\deg(f)$ and $s:=\deg(g)$ satisfy $r,s>1$.
Assume that, for every $n\in\N$, the polynomial $f\iter{n}(X)-g\iter{n}(Y)$
has a Siegel factor in $\C[X,Y]$.
Then either $f$ and $g$ have a common iterate or there is a linear
$\ell\in\C[X]$ such that
$(\ell\circ f\circ\ell\iter{-1},\ell\circ g\circ\ell\iter{-1})$
has one of the forms \eqref{cyc} or \eqref{cheb}.
\end{prop}

\begin{proof}
This follows from Propositions~\ref{gensec4} and \ref{gensec3}.
\end{proof}

\begin{cor}
\label{prop number field case III}
Pick $x,y\in\C$ and nonlinear $f,g\in \C[X]$.
If $\OO_f(x)\cap \OO_g(y)$ is infinite, then either
$f$ and $g$ have a common iterate or there is a
a linear $\ell\in\C[X]$ such that
$(\ell\circ f\circ\ell\iter{-1},\ell\circ g\circ\ell\iter{-1})$
has one of the forms \eqref{cyc} or \eqref{cheb}.
\end{cor}

\begin{proof}
Let $R$ be the ring generated by $x,y$ and the coefficients of
$f$ and $g$, and let $K$ be the field of fractions of $R$.
Note that both $R$ and $K$ are finitely generated.
Since $\OO_f(x)\cap\OO_g(y)$
is infinite, for each $n\in\N$ the equation $f\iter{n}(X)=g\iter{n}(Y)$ has
infinitely many solutions in $\OO_f(x)\times\OO_g(y)\subseteq R\times R$.
By Siegel's theorem (Corollary~\ref{siegel}), for each $n\in\N$ the polynomial
$f\iter{n}(X)-g\iter{n}(Y)$ has a Siegel factor in $K[X,Y]$.
Now the conclusion follows from the previous result (note that $f$ and $g$
are nonconstant since $\OO_f(x)$ and $\OO_g(y)$ are infinite).
\end{proof}

\begin{proof}[Proof of Theorem~\ref{mainthm}]
By Corollary~\ref{prop number field case III}, it suffices to prove
Theorem~\ref{mainthm} in case there is a linear $\ell\in\C[X]$ for which
$(\tilde f,\tilde g):=(\ell\circ f\circ\ell\iter{-1},\ell\circ g\circ\ell\iter{-1})$
has one of the forms \eqref{cyc} or \eqref{cheb}.  But then
\[
\OO_{\tilde f}(\ell(x)) \cap \OO_{\tilde g}(\ell(y)) =
\ell(\OO_f(x)) \cap \ell(\OO_g(y)) =
\ell(\OO_f(x) \cap \OO_g(y))
\]
is infinite, so Proposition~\ref{special case works} implies that
$\tilde f\iter{i}=\tilde g\iter{j}$ for some $i,j\in\N$, whence $f\iter{i}=g\iter{j}$.
\end{proof}

\begin{prop}
\label{special case works}
Pick $f,g\in\C[X]$ such that $(f,g)$ has one of the forms \eqref{cyc}
or \eqref{cheb}.
If there are $x,y\in\C$ for which $\OO_f(x)\cap\OO_g(y)$ is infinite,
then $f$ and $g$ have a common iterate.
\end{prop}

\begin{proof}
Assuming $\OO_f(x)\cap\OO_g(y)$ is infinite,
let $M$ be the set of pairs $(m,n)\in\N\times\N$ for which
$f\iter{m}(x)=g\iter{n}(y)$.  Note that any two elements of $M$ have
distinct first coordinates, since if $M$ contains $(m,n_1)$ and $(m,n_2)$
with $n_1\ne n_2$ then $g\iter{n_1}(y)=g\iter{n_2}(y)$ so $\OO_g(y)$ would
be finite.  Likewise, any two elements of $M$ have distinct second
coordinates, so there are elements $(m,n)\in M$ in which
$\min(m,n)$ is arbitrarily large.

Suppose $(f,g)$ has the form \eqref{cyc}.  Since $f\iter{m}(x)=x^{r^m}$
and $\OO_f(x)$ is infinite, $x$ is neither zero nor a root of unity.
We compute
\[
g\iter{n}(y)=\beta^{\frac{s^n-1}{s-1}}y^{s^n};
\]
putting $y_1:=\beta_1 y$ where $\beta_1\in\C^*$ satisfies
$\beta_1^{s-1}=\beta$, it follows that $g\iter{n}(y)=y_1^{s^n}/\beta_1$,
so infinitude of $\OO_g(y)$ implies that $y_1$ is neither zero nor a root
of unity.  A pair $(m,n)\in \N\times\N$ lies in $M$ if and only if
\begin{equation}
\label{extra 1st case pair}
x^{r^m} = \beta^{\frac{s^n-1}{s-1}}y^{s^n},
\end{equation}
or equivalently
\begin{equation}
\label{extra 1st case pair 2}
\beta_1x^{r^m}= y_1^{s^n}.
\end{equation}
Since \eqref{extra 1st case pair 2} holds for two pairs $(m,n)\in M$ which
differ in both coordinates, we have $x^a=y_1^b$ for some nonzero integers $a,b$.
By choosing $a$ to have minimal absolute value, it follows that the set
$S:=\{(u,v)\in\Z^2: \beta_1x^u=y_1^v\}$ has the form
$\{(c+ak,d+bk):k\in\Z\}$ for some $c,d\in\Z$.  For $(m,n)\in M$ we have
$(r^m,s^n)\in S$, so $(r^m-c)/a = (s^n-d)/b$.  Since $M$ is infinite,
Corollary~\ref{lang thm} implies that $c/a=d/b$.  In particular, every
$(m,n)\in M$ satisfies $br^m=as^n$.  Pick two
pairs $(m,n)$ and $(m+m_0,n+n_0)$ in $M$ with $m_0,n_0\in\N$.
Then $r^{m_0}=s^{n_0}$, and $S$ contains both
$(r^m,s^n)$ and $(r^{m+m_0},s^{n+n_0})$, so
\[
y_1^{s^n} x^{-r^m} = \beta_1 = y_1^{s^{n+n_0}} x^{-r^{m+m_0}},
\]
and thus
\[
(y_1^{s^n})^{s^{n_0}-1} = (x^{r^m})^{r^{m_0}-1}.
\]
Since $r^{m_0}=s^{n_0}$, it follows that $\beta_1^{s^{n_0}-1}=1$,
so $f\iter{m_0}=g\iter{n_0}$.

Now suppose $(f,g)$ has the form \eqref{cheb}.  Then (by Lemma~\ref{parity})
for any $m,n\in\N$ there exist $\epsilon_3,\epsilon_4\in\{1,-1\}$ such that
$(f\iter{m},g\iter{n})=(\epsilon_3T_{r^m},\epsilon_4T_{s^n})$.
Since $\OO_{f}(x)\cap\OO_g(y)$ is infinite, we can choose
$\delta\in\{1,-1\}$ such that $T_{r^m}(x)=\delta T_{s^n}(y)$
for infinitely many $(m,n)\in\N\times\N$.  Pick $x_0,y_0\in\C^*$ such that
$x_0+x_0^{-1}=x$ and $y_0+y_0^{-1}=y$.  Then there are infinitely many
pairs $(m,n)\in\N\times\N$ for which
\[
x_0^{r^m} + x_0^{-r^m} = \delta(y_0^{s^n} + y_0^{-s^n}),
\]
so we can choose $\epsilon\in\{1,-1\}$ such that
\begin{equation} \label{new}
x_0^{r^m} = \delta y_0^{\epsilon s^n}
\end{equation}
for infinitely many $(m,n)\in\N\times\N$.
Moreover, since $\OO_f(x)$ and $\OO_g(y)$ are infinite,
neither $x_0$ nor $y_0$ is a root of unity, so distinct pairs
$(m,n)\in\N\times\N$ which satisfy \eqref{new} must differ in both
coordinates.  Now \eqref{new} is a reformulation of
\eqref{extra 1st case pair 2}, so we conclude as above that
$r^{m_0}=s^{n_0}$ for some $m_0,n_0\in\N$ such that
$\delta^{s^{n_0}-1}=1$.  If $s$ is odd, it follows that
$f\iter{2m_0}=g\iter{2n_0}$.  If $s$ is even then we cannot have
$\delta=-1$; since $f\iter{m}=\epsilon_1 T_{r^m}$ and
$g\iter{n}=\epsilon_2 T_{s^n}$, it follows that $\epsilon_1=\epsilon_2$, so
$f\iter{m_0}=g\iter{n_0}$.
\end{proof}

\begin{remark}
If $(f,g)$ has the form \eqref{cyc} or \eqref{cheb}, then $f^n(X)-g^m(Y)$
has a Siegel factor in $\C[X,Y]$ for every $n,m\in\N$ (in fact, $f^n(X)-g^m(Y)$
is the product of irreducible Siegel polynomials).  So the results of the
previous two sections give no information.  To illustrate Theorem~\ref{mainthm}
for such $(f,g)$, consider $(f,g)=(X^2,X^3)$.  In this case, for any $n,m\in\N$,
the equation $f^n(X)=g^m(Y)$ has infinitely many solutions in $\Z\times\Z$.
However, for any $x_0,y_0\in\C$, each such equation has only finitely many
solutions in $\OO_f(x_0)\times\OO_g(y_0)$.
In particular, each such equation has only finitely many solutions
in $\OO_f(2)\times\OO_g(2)=\{(2^{2^a},2^{3^b}):a,b\in\N_0\}$,
but has infinitely many solutions in $2^{\N_0}\times 2^{\N_0}$.
The underlying principle is that orbits are rather thin subsets of $\C$.
\end{remark}


\section{A multivariate generalization}
\label{sec multivariate}

In this section we show that Theorem~\ref{mainthm} implies
Theorem~\ref{manyvariablethm} and Corollary~\ref{manyvariablecor}.

\begin{proof}[Proof of Theorem~\ref{manyvariablethm}.]
We use induction on $d$.  If $d=1$ then $L(\C)=\C$, so $f(L)=L$.
Now assume the result holds for lines in $\C^{d-1}$.

If all points of $L$ take the same value $z_d$ on the last
coordinate, then $L=L_0\times\{z_d\}$ for some line $L_0\subset\C^{d-1}$.
By the inductive hypothesis, there exist nonnegative integers
$m_1,\dots,m_{d-1}$ (not all zero) such that $L_0$ is invariant under
$(f_1\iter{m_1},\dots,f_{d-1}\iter{m_{d-1}})$.  Then $L$ is invariant under
$(f_1\iter{m_1},\dots,f_{d-1}\iter{m_{d-1}},f_d\iter{0})$, as desired.

Henceforth assume that $L$ projects surjectively onto each coordinate.
Then any point of $L$ is uniquely determined by its value at any
prescribed coordinate.  Since $L$ contains infinitely many points on
$\OO_{f_1}(x_1)\times\dots\times\OO_{f_d}(x_d)$,
it follows that $\OO_{f_i}(x_i)$ is infinite for each $i$.
For each $i=2,\dots,d$, let $\pi_i\colon\C^d\to\C^2$ be the projection onto
the first and $i\tth$ coordinates of $\C^d$.  Then $L_i:=\pi_i(L)$ is a line
in $\C^2$ having infinite intersection with $\OO_{f_1}(x_1)\times\OO_{f_i}(x_i)$.
Since $L$ projects surjectively onto each
coordinate, $L_i$ is given by the equation $X_i=\sigma_i(X_1)$ for some
degree-one $\sigma_i\in\C[X]$.  For any $k,\ell\in\N$ such that
\[
(f_1\iter{k}(x_1),f_i\iter{\ell}(x_i))\in L_i,
\] 
we have $(\sigma_i\circ f_1\circ\sigma_i\iter{-1})\iter{k}(\sigma_i(x_1)) =
f_i\iter{\ell}(x_i)$.  Thus, by Theorem~\ref{mainthm} there exist
$m_i,n_i\in\N$ such that
\[
(\sigma_i\circ f_1 \circ\sigma_i\iter{-1})\iter{m_i}=f_i\iter{n_i}.
\]
Let $M_1$ be the least common multiple of all the $m_i$, and for each
$i\ge 2$ define $M_i:=(n_iM_1)/m_i$.  Then
\[
(\sigma_i\circ f_1 \circ\sigma_i\iter{-1})\iter{M_1}=f_i\iter{M_i},
\]
so for any $y_1\in\C$ we have
\[
f_i\iter{M_i}(\sigma_i(y_1)) = \sigma_i\circ f_1\iter{M_1}(y_1).
\]
Since $L$ is defined by the $(d-1)$ equations $X_i=\sigma_i(X_1)$, it follows
that $L$ is invariant under $(f_1\iter{M_1},\dots,f_d\iter{M_d})$.
\end{proof}

\begin{proof}[Proof of Corollary~\ref{manyvariablecor}.]
Arguing inductively as in the above proof, we may assume that the projection
of $L$ onto each coordinate of $\C^d$ is surjective.  Thus each point of $L$
is uniquely determined by its value on any prescribed coordinate.
By Theorem~\ref{manyvariablethm}, $L$ is preserved by
$\rho_1^{m_1}\dots\rho_d^{m_d}$ for some nonnegative integers
$m_1,\dots,m_d$ which are not all zero.  Without loss of generality, assume $m_1>0$.
For each $k$ with $1\le k\le m_1$, let $U_k$ be the set of tuples
$(n_1,\dots,n_d)\in (\N_0)^d$ such that $n_1\equiv k\pmod{m_1}$
and $\rho_1^{n_1}\dots\rho_d^{n_d}(\alpha)$ lies on $L$.
If $U_k$ is
nonempty, pick $(n_1,\dots,n_d)\in U_k$ for which $n_1$ is minimal; then
$U_k$ contains $V_k:=\{(n_1+jm_1,\dots,n_d+jm_d): j\in\N_0\}$,
and the set $Z_k$ of values $\rho_1^{u_1}\dots\rho_d^{u_d}(\alpha)$ for
$(u_1,\dots,u_d)\in U_k$ is the same as the corresponding set for
$(u_1,\dots,u_d)\in V_k$.  Thus $Z_k$ is the orbit of $\alpha$ under
$\langle \rho_1^{m_1}\dots\rho_d^{m_d}\rangle \rho_1^{n_1}\dots\rho_d^{n_d}$,
which is a coset of a cyclic subsemigroup.
\end{proof}


\section{Function field case, second proof}
\label{second proof for function fields 0}

We now turn our attention to the following result.
\begin{thm}
\label{first extension}
Let $K$ be a field of characteristic $0$, let $f,g\in K[X]$ be polynomials of
degree greater than one, and let $x_0,y_0\in K$. Assume there is no linear
$\mu\in \Kbar[X]$ for which $\mu\iter{-1}(x_0),\mu\iter{-1}(y_0)\in\Qbar$
and both $\mu\iter{-1}\circ f\circ \mu$ and
$\mu\iter{-1}\circ g\circ\mu$ are in $\Qbar[X]$.
If $\OO_f(x_0)\cap\OO_g(y_0)$ is
infinite, then $f$ and $g$ have a common iterate.
\end{thm}

Theorem~\ref{first extension} may be viewed as the `function field' part of our
Theorem~\ref{mainthm}.
We will give an alternate proof of Theorem~\ref{first extension} using the
theory of heights.  In the next two sections we review canonical
heights associated to nonlinear polynomials.  Then in
Section~\ref{second proof for function fields} we will prove
Theorem~\ref{first extension} by reducing it to the case $\deg(f)=\deg(g)$
handled in our previous paper \cite[Thm.\ 1.1]{lines}.  Here we avoid the
intricate arguments about polynomial decomposition used in the first part of
the present paper;
instead our proof relies on a result of Lang,
already used in the proof of Proposition~\ref{special case works}, which is
itself a consequence of Siegel's theorem.


\section{Canonical heights associated to polynomials}
\label{trivial action}

In this section we recall some standard terminology about heights.
First, a \emph{global field} is either a number field or a function field of
transcendence degree $1$ over another field.  Any global field $E$ comes
equipped with a standard set $M_E$ of absolute values $|\cdot|_v$ which
satisfy a product formula

\[
\prod_{v\in M_E} |x|^{N_v}_v = 1\quad\text{ for every $x\in E^*$},
\]
where $N\colon M_E\to\N$ and $N_v:=N(v)$ (cf.\ \cite{Lang_diophantine} for
details).

If $E$ is a global field, the logarithmic Weil height of $x\in\Ebar$
(with respect to $E$) is defined as (see \cite[p.\ $52$]{Lang_diophantine})
\[
  h_E(x) = \frac{1}{[E(x):E]}\cdot \sum_{v\in M_E}
   \sum_{\substack{w|v\\ w\in M_{E(x)}}} \log\max\{|x|^{N_w}_w , 1\}.
\]

\begin{defi}
\label{definition canonical height}
Let $E$ be a global field, let $\phi\in E[X]$ with $\deg(\phi)>1$,
and let $z\in \Ebar$. The
canonical height $\hh_{\phi,E}(z)$ of $z$ with respect to $\phi$ (and $E$) is
\[
\widehat{h}_{\phi,E}(z) :=
 \lim_{k\rightarrow\infty} \frac{h_E(\phi\iter{k}(z))}{\deg(\phi)^k}.
\]
\end{defi}

Call and Silverman \cite[Thm.\ 1.1]{Call-Silverman} proved the
existence of the above limit, using boundedness of
$|h_E(\phi(x)) - (\deg \phi) h_E(x)|$ and a
telescoping sum argument due to Tate.
We will usually write $h(x)$ and $\hh_{\phi}(x)$ rather than
$h_E(x)$ and $\hh_{\phi,E}(x)$; this should not cause confusion.
We will use the following properties of the canonical height.
\begin{prop}\label{can}
Let $E$ be a global field, let $\phi\in E[X]$ be a polynomial of degree
greater than $1$, and let $z\in \Ebar$. Then
\begin{itemize} 
\item[(a)] for each $k\in\N$, we have
$\hh_{\phi}(\phi\iter{k}(z)) = \deg(\phi)^k \cdot \hh_{\phi}(z)$;
\item[(b)] $|h(z) - \hh_{\phi}(z)|$ is bounded by a function which does not
  depend on $z$;
\item[(c)] if $E$ is a number field then $z$ is preperiodic if and only if
 $\hh_{\phi}(z) = 0$.
\end{itemize}
\end{prop}

\begin{proof}
  Part $(a)$ is clear; for $(b)$ see \cite[Thm.\ $1.1$]{Call-Silverman};
  and for $(c)$ see \cite[Cor.\ $1.1.1$]{Call-Silverman}.
\end{proof}
Part $(c)$ of Proposition~\ref{can} is not true if $E$ is a function
field with constant field $E_0$, since $\hh_{\phi}(z)=0$ whenever $z\in E_0$
and $\phi\in E_0[X]$.  But these are essentially the only counterexamples
in the function field case (cf.\ Lemma~\ref{Benedetto}).


\section{Canonical heights in function fields}
\label{Silverman proof}

The setup for this section is as follows: $E$ is a field, and $K$ is a
function field of transcendence degree $1$ over $E$.

First we note that for each place $v\in M_K$ of the function field $K$, we may
assume $\log |z|_v\in \Q$ (we use $c:=e^{-1}$ in the definition of absolute
values on function fields from \cite[p.\ $62$]{Lang_diophantine}).

Let $\phi\in K[X]$ be a polynomial of degree greater than $1$. For each
$v\in M_K$, we let 
\begin{equation}
\label{local canonical height}
\hh_{\phi,v}(z):=
\lim_{n\to\infty}\frac{\log \max\{ |\phi\iter{n}(z)|^{N_v}_v,1\}}{\deg(\phi)^n}
\end{equation}
be the canonical local height of $z\in K$ at $v$. Clearly, for all but
finitely many $v\in M_K$, all coefficients of $\phi$, and $z$ are $v$-adic
integers. Hence, for such $v\in M_K$, we have $\hh_{\phi,v}(z) = 0$. Moreover,
it is immediate to show that
\begin{equation}
\label{local canonical height sum}
\hh_{\phi}(z) = \sum_{v\in M_K} \hh_{\phi,v}(z).
\end{equation}
For a proof of the existence of the limit in \eqref{local canonical height},
and of the equality in \eqref{local canonical height sum}, see
\cite{Call-Goldstine}.

The following result is crucial for Section~\ref{second proof for function
fields}.
\begin{lemma}
\label{rational heights}
For each $z\in K$, and for each $\phi\in K[X]$ with $d:=\deg(\phi)>1$, we have
$\hh_{\phi}(z)\in\Q$.
\end{lemma}

\begin{proof}
For each $v\in M_K$, there exists $M_v > 0$ such that
$\hh_{\phi,v}(z) > 0$ if and only if there exists $n\in \N$ such that
$|\phi\iter{n}(z)|_v > M_v$, and moreover, in this case 
\[
\hh_{\phi,v}(\phi\iter{n}(z)) = \log |\phi\iter{n}(z)|_v +
 \frac{ \log|\delta_{d}|_v}{d -1},
\]
where $\delta_{d}$ is the leading coefficient of $\phi$. For a proof of this
claim, see \cite[Lemma 4.4]{siegel} (actually, in \cite{siegel} the above
claim is proved only for Drinfeld modules, but that proof works identically
for all polynomials defined over a function field in any characteristic).

We claim that the above fact guarantees that $\hh_{\phi,v}(z) \in \Q$. 
Indeed, if $\hh_{\phi,v}(z) > 0$, then there exists $n\in \N$ such that
$|\phi\iter{n}(z)|_v > M_v$. So, 
\begin{equation}
\hh_{\phi,v}(z) = \frac{\hh_{\phi,v}(\phi\iter{n}(z))}{d^n} 
= \frac{\log |\phi\iter{n}(z)|_v +\frac{\log|\delta_{d}|_v}{d -1}}{d^n} \in \Q.
\end{equation}
Since $\hh_{\phi}(z)$ is the sum of finitely many local heights
$\hh_{\phi,v}(z)$, we conclude that $\hh_{\phi}(z)\in \Q$.
\end{proof}

The following
result about canonical heights of non-preperiodic points for
non-\emph{isotrivial} polynomials will be used later.
\begin{defi}
\label{isotrivial}
We say a polynomial $\phi \in K[X]$ is isotrivial over $E$ if there
exists a linear $\ell
\in \Kbar[X]$ such that $\ell\circ \phi \circ \ell\iter{-1} \in \Ebar[X]$.
\end{defi}
Benedetto proved that a non-isotrivial polynomial has nonzero canonical height
at its nonpreperiodic points \cite[Thm.\ B]{Bene}:

\begin{lemma}
\label{Benedetto}
  Let $\phi\in K[X]$ with $\deg(\phi)\ge 2$, and let
$z\in \overline{K}$. If $\phi$ is non-isotrivial over $E$, then
$ \hh_{\phi}(z) = 0$ if and only if $z$ is preperiodic for~$\phi$.
\end{lemma}

We state one more preliminary result, which is proved in
\cite[Lemma 6.8]{lines}.
\begin{lemma}
\label{height 0 for isotrivial}
Let $\phi\in K[X]$ be isotrivial over $E$, and let $\ell$ be as
in Definition~\ref{isotrivial}.  If $z\in\overline{K}$ satisfies
$\hh_{\phi}(z) = 0$, then $\ell(z)\in\Ebar$.
\end{lemma}

\begin{defi}
\label{isotriviality with a point}
With the notation as in Lemma~\ref{height 0 for isotrivial}, we call
the pair $(\phi,z)$ isotrivial. Furthermore, if $F\subset K$ is any
subfield, and there exists a linear polynomial $\ell\in\Kbar[X]$ such that
$\ell\circ\phi\circ\ell\iter{-1} \in \Fbar[X]$ and
$\ell(z)\in \Fbar$,
then we call the pair $(\phi,z)$ isotrivial over $F$.
\end{defi}


\section{Proof of Theorem~\ref{first extension}}
\label{second proof for function fields}

We first prove two easy claims.
\begin{claim}
\label{key claim 0}
Let $E$ be any subfield of $K$, and assume that $(f,x_0)$ and $(g,y_0)$ are
isotrivial over $E$. If $\OO_f(x_0)\cap\OO_g(y_0)$ is
infinite, then there exists a linear $\mu\in\Kbar[X]$ such that
$\mu\circ f\circ \mu\iter{-1},\mu\circ g\circ \mu\iter{-1}\in \Ebar[X]$ and
$\mu(x_0),\mu(y_0)\in \Ebar$.
\end{claim}

\begin{proof}[Proof of Claim~\ref{key claim 0}.]
  We know that there exist linear $\mu_1,\mu_2\in \Kbar[X]$ such that
  $f_1:=\mu_1\circ f\circ\mu_1\iter{-1}\in\Ebar[X]$ and
  $g_1:=\mu_2\circ g\circ \mu_2\iter{-1}\in\Ebar[X]$, and
  $x_1:=\mu_1(x_0)\in\Ebar$ and $y_1:=\mu_2(y_0)\in\Ebar$.  Thus
  $\OO_{f_1}(x_1)=\mu_1(\OO_f(x_0))$ and
  $\OO_{g_1}(y_1)=\mu_2(\OO_g(y_0))$.  Since
  $\OO_f(x_0)\cap\OO_g(y_0)$ is infinite, there are infinitely many
  pairs $(z_1,z_2)\in\Ebar\times\Ebar$ such that
  $\mu_1\iter{-1}(z_1)=\mu_2\iter{-1}(z_2)$.  Thus
  $\mu:=\mu_2\circ\mu_1\iter{-1}\in\Ebar[X]$.  Hence
\[
\mu_1\circ g\circ\mu_1\iter{-1}=
\mu\iter{-1}(\mu_2\circ g\circ\mu_2\iter{-1})\mu\in\Ebar[X],
\]
and 
\[
\mu_1(y_0)=(\mu_1\circ\mu_2\iter{-1})(y_1)=\mu\iter{-1}(y_1)\in\Ebar,
\]
as desired.
\end{proof}

\begin{claim}
\label{key claim}
If $\OO_f(x_0)\cap\OO_g(y_0)$ is infinite, then there exist subfields
$E\subset F\subset K$ such that $F$ is a function field of
transcendence degree $1$ over $E$, and there exists a linear
polynomial $\mu\in \Kbar[X]$ such that $\mu\circ f\circ
\mu\iter{-1},\mu\circ g\circ\mu\iter{-1}\in \Fbar[X]$, and
$\mu(x_0),\mu(y_0)\in\Fbar$, and either $(f,x_0)$ or $(g,y_0)$ is not
isotrivial over $E$.
\end{claim}

\begin{proof}[Proof of Claim~\ref{key claim}.]
Let $K_0$ be a finitely generated subfield of $K$ such that $f,g\in K_0[X]$ and
$x_0,y_0\in K_0$. Then there exists a finite tower of field subextensions:
\[
K_s\subset K_{s-1}\subset\dots\subset K_1\subset K_0
\]
such that $K_s$ is a number field, and for each $i=0,\dots,s-1$, the extension
$K_i/K_{i+1}$ is finitely generated of transcendence degree $1$. Using
Claim~\ref{key claim 0} and the hypotheses of Theorem~\ref{first extension}, we
conclude that there exists $i=0,\dots,s-1$, and there exists a linear
$\mu\in\bar{K_0}[X]$ such that
$\mu\circ f\circ\mu\iter{-1},\mu\circ g\circ\mu\iter{-1}\in \bar{K_i}[X]$, and
$\mu(x_0),\mu(y_0)\in \bar{K_i}$, and either $(f,x_0)$ or $(g,y_0)$ is not
isotrivial over $K_{i+1}$.
\end{proof}

\begin{proof}[Proof of Theorem~\ref{first extension}.]
Let $E$, $F$ and $\mu$ be as in the conclusion of Claim~\ref{key claim}. At the
expense of replacing $f$ and $g$ with their respective conjugates by $\mu$, and
at the expense of replacing $F$ by a finite extension, we may assume that
$f,g\in F[X]$, and $x_0,y_0\in F$, and $(f,x_0)$ is not isotrivial over $E$. 

Let $d_1:=\deg(f)$ and $d_2:=\deg(g)$. We construct the canonical heights
$\hh_f$ and $\hh_g$ associated to the polynomials $f$ and $g$, with respect to
the set of absolute values associated to the function field $F/E$.
Because $(f,x_0)$ is non-isotrivial, and because $x_0$ is not
preperiodic for $f$ (note that $\OO_f(x_0)\cap\OO_g(y_0)$ is infinite),
Lemma~\ref{Benedetto} yields that
$H_1 := \hh_f(x_0) > 0$. Moreover, if $H_2:=\hh_g(y_0)$, then using
Lemma~\ref{rational heights}, we have that $H_1,H_2\in \Q$. 
Because there exist infinitely many pairs $(m,n)\in\N\times \N$ such that
$f\iter{m}(x_0) = g\iter{n}(y_0)$, Proposition~\ref{can} $(a)-(b)$ yields that
\begin{equation}
\label{first boundedness}
|d_1^{m} \cdot H_1 - d_2^{n} \cdot H_2| \text{ is bounded}
\end{equation}
for infinitely many pairs $(m,n)\in \N\times \N$. Because
$H_1, H_2 \in \Q$, we conclude that there exist \emph{finitely} many
rational numbers $\gamma_1,\dots,\gamma_s$ such that
\[
\gamma_i=d_1^{m} \cdot H_1 - d_2^{n} \cdot H_2
\]
for each pair $(m,n)$ as in \eqref{first boundedness}. (We are using the fact
 that there are finitely
many rational numbers of bounded denominator, and bounded absolute value.)
Therefore, there exists a rational number
$\gamma:=\gamma_i$ (for some $i=1,\dots,s$) such that 
\begin{equation}
\label{second boundedness}
d_1^{m} H_1 - d_2^{n} H_2 = \gamma.
\end{equation}
for infinitely many pairs $(m,n)\in \N\times \N$. Hence, the line
$L\subset \A^2$ given by the
equation $H_1 \cdot X - H_2 \cdot Y = \gamma$ has infinitely many
points in common with the rank-$2$ subgroup
$\Gamma := \{(d_1^{k_1},d_2^{k_2}) \text{ : } k_1,k_2\in \mathbb{Z}\}$ of
$\mathbb{G}_m^2$.
Using Corollary~\ref{lang thm}, we obtain that $\gamma=0$.
 Because there are infinitely many pairs $(m,n)$ satisfying
 \eqref{second boundedness}, and because $H_1\ne 0$, we conclude that
 there exist positive integers $m_0$ and $n_0$ such that
 $d_1^{m_0} = d_2^{n_0}$; thus $\deg(f\iter{m_0}) = \deg(g\iter{n_0})$. Because
$\OO_f(x_0)\cap\OO_g(y_0)$ is infinite, we can find $k_0,\ell_0\in\N$ such that
$\OO_{f\iter{m_0}}(f\iter{k_0}(x_0))\cap\OO_{g\iter{n_0}}(g\iter{\ell_0}(y_0))$
is infinite. Because $\deg(f\iter{m_0})=\deg(g\iter{n_0})$, we can apply
\cite[Thm.\ $1.1$]{lines} and conclude the proof of
Theorem~\ref{first extension}.
\end{proof}

\begin{remark}
Theorem~\ref{first extension} holds essentially by the same argument as above,
if $x_0$ is not in the $v$-adic filled Julia set of $f$, where $v$ is any place
of a function field $K$ over a field $E$.
\end{remark}

\begin{remark}
One can show that if $f$ is a linear polynomial, and $g$ is any nonisotrivial
polynomial of degree larger than one, then $\OO_f(x_0)\cap\OO_g(y_0)$ is
finite. This assertion fails if $g$ is isotrivial, as shown by the infinite
intersection $\OO_{X+1}(0)\cap\OO_{X^2}(2)$.
\end{remark}


\section{The dynamical Mordell--Lang problem}
\label{S:ML}

In this section we discuss topics related to Question~\ref{semigroup}.
We give examples where this question has a negative answer, and we show that
the Mordell--Lang conjecture can be reformulated as a particular instance of
Question~\ref{semigroup}.  Then we discuss the connection between
Question~\ref{semigroup} and the existence of invariant subvarieties, and
the connection between this question, Zhang's conjecture, and critically
dense sets.

\subsection{Examples}

There are several situations where Question~\ref{semigroup} has a negative
answer.  Let $\Phi(x,y)=(2x,y)$ and $\Psi(x,y)=(x,y^2)$ be endomorphisms of
$\A^2$; let $S$ be the semigroup generated by $\Phi$ and $\Psi$.  If $\Delta$
is the diagonal subvariety of $\A^2$, then
$\Delta(\C)\cap \OO_S((1,2))=\{\Phi^{2^n}\Psi^n((1,2))\text{ : }n\in\N_0\}$,
which yields a negative answer to Question~\ref{semigroup}.  A similar example
occurs for $X=E\times E$ with $E$ any commutative algebraic group, where
$\Phi(P,Q)=(P+P_0,Q)$ and $\Psi(P,Q)=(P,2Q)$ with $P_0\in E(\C)$ a nontorsion
point: letting $\Delta$ be the diagonal in $E\times E$, and $S$ the semigroup
generated by $\Phi$ and $\Psi$, we have $\Delta(\C)\cap \OO_S((0,P_0))=
\{\Phi^{2^n}\Psi^n((0,P_0))\text{ : }n\in\N_0\}$ (where $0$ is the identity
element of the group $E(\C)$).
One can produce similar examples in which $S$ contains infinite-order elements
which restrict to automorphisms on some positive-dimensional subvariety of $X$.
However, there is an important situation where $S$ consists of automorphisms
but Question~\ref{semigroup} has an affirmative answer, namely when $S$
consists of translations on a semiabelian variety $X$; we discuss this below.


\subsection{Mordell--Lang conjecture}

We show that the Mordell--Lang conjecture is a particular case of our
Question~\ref{semigroup}.  This conjecture, proved by Faltings \cite{Faltings}
and Vojta \cite{V1}, describes the intersection of subgroups and subvarieties
of certain algebraic groups:

\begin{thm} \label{ML}
Let $X$ be a semiabelian variety over $\C$, let $V$ be a subvariety, and let
$\Gamma$ be a finitely generated subgroup of $X(\C)$.  Then $V(\C)\cap \Gamma$
is the union of finitely many cosets of subgroups of $\Gamma$.
\end{thm}

Here a \emph{semiabelian variety} is a connected algebraic group $X$ which
admits an exact sequence $1 \to \G_m^k \to X \to A \to 1$ with $A$ an abelian
variety and $k\in\N_0$.  Any such $X$ is commutative.

Let $X$ be a semiabelian variety over $\C$, let $\Gamma$ be the subgroup of
$X(\C)$ generated by $P_1,\dots,P_r\in X(\C)$, let $\tau_i$ be the
translation-by-$P_i$ map on $X$ for each $i=1,\dots,r$, and let
$S:=\langle\tau_1,\dots,\tau_r\rangle$ be the finitely generated commutative
semigroup generated by the translations $\tau_i$.  Let $\overline{S}$ be
the group generated by the automorphisms $\tau_i$ for $i=1,\dots,r$;
thus $\Gamma=\OO_{\overline{S}}(0)$.
Plainly Theorem~\ref{ML} implies an affirmative answer to
Question~\ref{semigroup}(b).  Conversely, Theorem~\ref{ML} follows from
Question~\ref{semigroup}(b) applied to the semigroup generated by all
translations $\pm \tau_i$.  It can be shown that Theorem~\ref{ML} also follows
quickly from Question~\ref{semigroup}(a) applied to the semigroups $S$ and
$S^{-1}$.

\subsection{Invariant subvarieties}

Suppose Question~\ref{semigroup} has an affirmative answer for some
$X$, $V$, $S$, and $\alpha$.  Then $V(\C)\cap \OO_S(\alpha)$ is the
union of finitely many sets of the form $T_0:=\OO_{S_0\cdot\Phi}(\alpha)$,
with $\Phi\in S$ and $S_0$ a subsemigroup of $S$.  For any such $T_0$,
let $V_0$ be the Zariski closure of $T_0$, so $V_0\subset V$ and
$S_0(V_0)\subset V_0$ (since $S_0(T_0)\subset T_0$).  Thus, the Zariski
closure of $V(\C)\cap\OO_S(\alpha)$ consists of finitely many points and
finitely many positive-dimensional subvarieties $V_0\subset V$, where for
each $V_0$ there is an infinite subsemigroup $S_0$ of $S$ such that
$S_0(V_0)\subset V_0$.

Conversely, if the Zariski closure of $V(\C)\cap\OO_S(\alpha)$ has this form,
and if each $S_0$ has finite index in $S$ (as happens, for instance, if $S$ is
cyclic), then Question~\ref{semigroup}(a) has an
affirmative answer.  We do not know whether this implication remains true in
general when $S_0$ has infinite index in $S$.

\subsection{Zhang's conjecture and critically dense sets}

Zhang considers the action of an endomorphism $\Phi$ of an
irreducible projective variety $X$ over a number field $K$, under the
hypothesis that $\Phi$ is \emph{polarizable} in the sense that
$\Phi^*\scL\simeq\scL^q$ for some line bundle $\scL$ and some $q>1$.
Zhang conjectures that $\OO_{\Phi}(\alpha)$ is Zariski dense in $X$ for some
$\alpha\in X({\overline K})$ \cite[Conj.~4.1.6]{ZhangLec}.
Let $Y$ be the union of all proper subvarieties $V$ of $X$ which are
$\Phi$-\emph{preperiodic} (i.e., $\Phi^{k+N}(V)=\Phi^k(V)$
for some $k\ge 0$ and $N\ge 1$).  We now show that $Y(\Kbar)$ consists of the
points $\alpha\in X(\Kbar)$ for which $\OO_{\Phi}(\alpha)$ is not Zariski
dense in $X$; thus Zhang's conjecture amounts to saying $X\ne Y$.

Pick $\alpha\in Y(\Kbar)$, and let $V\subset Y$ be a proper $\Phi$-preperiodic
subvariety of $X$ such that $\alpha\in V(\Kbar)$; moreover, pick $k\ge 0$ and
$N\ge 1$ such that $\Phi^{k+N}(V)=\Phi^k(V)$.  Then
$\OO_{\Phi^N}(\Phi^k(\alpha))\subset \Phi^k(V)$, so
\[
\OO_{\Phi}(\alpha)\subset \bigcup_{i=0}^{k+N-1} \Phi^i(V).
\]
Since $V\ne X$ and $X$ is irreducible, it follows that $\OO_{\Phi}(\alpha)$ is
not Zariski dense in $X$.

Conversely, pick $\alpha\in X(\Kbar)\setminus Y(\Kbar)$, and let $Z$ be
the Zariski closure of $\OO_{\Phi}(\alpha)$.  One can show that polarizable
endomorphisms are closed, so $\Phi^n(Z)$ is a closed subvariety of $X$ for
each $n\ge 1$.
Since $\Phi^n(\OO_{\Phi}(\alpha))\subset \Phi^{n-1}(\OO_{\Phi}(\alpha))$, it
follows that $\Phi^n(Z)\subset \Phi^{n-1}(Z)$.
Hence $Z\supset \Phi(Z)\supset \Phi^2(Z)\supset\dots$ is a descending chain
of closed subvarieties of $X$, so $\Phi^{N+1}(Z)=\Phi^N(Z)$ for some $N\ge 0$,
whence $Z$ is $\Phi$-preperiodic.  Since $\alpha\notin Y(\Kbar)$, it follows
that $Z=X$.

If we replace $K$ by $\C$, we suspect Zhang's conjecture holds even without the
polarizability condition, and also if $X$ is allowed to be quasiprojective.
Let $Y$ be the union of the proper subvarieties $V$ of $X$ for which there
exists $N\in\N$ with $\Phi^N(V)\subset V$.  The above argument shows that
$Y(\C)$ consists of the points $\alpha\in X(\C)$ for which
$\OO_{\Phi}(\alpha)$ is not Zariski
dense in $X$.  If $\Phi$ is a closed morphism (as in the case of Zhang's
polarizable endomorphisms), then each subvariety $V$ for which
$\Phi^N(V)\subset V$ is actually $\Phi$-preperiodic.

On the other hand, a positive answer to our Question~\ref{semigroup} yields
that each Zariski dense orbit $\OO_{\Phi}(\alpha)$ intersects any proper
subvariety $V$ of the irreducible quasiprojective variety $X$ in at most
finitely many points.  Indeed, if $\OO_{\Phi}(\alpha)\cap V(\C)$ were infinite,
then there exists $k,N\in \N$ such that
$\OO_{\Phi^N}(\Phi^k(\alpha))\subset V(\C)$.  Therefore
\[
\OO_{\Phi}(\alpha)\subset \{\Phi^i(\alpha)\text{ : } 0\le i\le k-1\}
\bigcup\left(\cup_{j=0}^{N-1} \Phi^j(V)\right),
\]
and since $\dim(V)<\dim(X)$ it follows that $\OO_{\Phi}(\alpha)$ is not Zariski
dense in $X$.

Thus, if Question~\ref{semigroup} has a positive answer for an
irreducible quasiprojective variety $X$, then any Zariski dense orbit
$\OO_{\Phi}(\alpha)$ is \emph{critically dense}, in the terminology of
\cite[Def.\ 3.6]{KeeRogSta} and \cite[\S 5]{SriCut}:
\begin{defi}
\label{critically dense}
Let $U$ be an infinite set of closed points of an integral scheme $X$. Then we
say that $U$ is critically dense if every infinite subset of $U$ has Zariski
closure equal to $U$.
\end{defi}

\end{document}